\documentclass[10pt]{article} 

\usepackage[utf8]{inputenc} 
\usepackage{amsmath}
\usepackage{amssymb}
\usepackage{amsfonts}
\usepackage{amsthm}
\usepackage{authblk}
\usepackage{bbm}
\usepackage{pgf, tikz}
\usetikzlibrary{shapes}

\usepackage{geometry} 
\usepackage{graphicx} 

\usepackage{booktabs} 
\usepackage{array} 
\usepackage{paralist} 
\usepackage{verbatim} 
\usepackage{subfig} 

\usepackage{fancyhdr} 
\usepackage{sectsty}
\allsectionsfont{\sffamily\mdseries\upshape} 

\usepackage[nottoc,notlof,notlot]{tocbibind} 
\usepackage[titles,subfigure]{tocloft} 

\newtheorem{theorem}{Theorem}[section]
\newtheorem{proposition}[theorem]{Proposition}

\newtheorem{conjecture}[theorem]{Conjecture}
\newtheorem{corollary}[theorem]{Corollary}
\newtheorem{lemma}[theorem]{Lemma}

\newtheorem{question}[theorem]{Question}

\theoremstyle{definition}
\newtheorem{definition}[theorem]{Definition}
\newtheorem{remark}[theorem]{Remark}

\renewcommand{\a}{{\alpha}}
\renewcommand{\b}{{\beta}}
\newcommand{\g}{{\gamma}}
\renewcommand{\d}{{\delta}}
\newcommand{\e}{{\varepsilon}}
\newcommand{\n}{{\nu}}
\renewcommand{\L}{{\Lambda}}
\renewcommand{\l}{{\lambda}}
\renewcommand{\r}{{\rho}}
\newcommand{\s}{{\sigma}}

\newcommand{\E}{{\mathbb E}}

\renewcommand{\P}{{\mathbb P}}
\newcommand{\Z}{{\mathbb Z}}

\renewcommand{\O}{{\mathcal{O}}}

\newcommand{\Aut}{{\mathrm{Aut}}}

\renewcommand{\H}{{\mathrm H}}
\newcommand{\unif}{{\mathrm{unif}}}
\newcommand{\diag}{{\mathrm{diag}}}
\newcommand{\diam}{{\mathrm{diam}}}
\newcommand{\supp}{{\mathrm{supp}}}
\newcommand{\Loc}{{\mathrm{Loc}}}
\newcommand{\Sym}{{\mathcal{S}}}
\newcommand{\id}{{\mathrm{id}}}
\newcommand{\Coup}{{\mathrm{Coup}}}
\newcommand{\word}{{\mathrm{word}}}

\renewcommand{\u}{{\mathbf{u}}}

\title{Noise sensitivity of random walks on groups}
\author[1]{Ita\"i Benjamini}
\author[2]{J\'er\'emie Brieussel}

\affil[1]{{\small Department of Mathematics, Weizmann Institute of Science, Rehovot, Israel}}
\affil[2]{{\small Institut Montpellierain Alexander Grothendieck, Universit\'e de Montpellier, France}}

\begin{document}
\maketitle

\begin{abstract}
A random walk on a group is noise sensitive if resampling every step independently with a small probability results in an almost independent output. We precisely define two notions: $\ell^1$-noise sensitivity and entropy noise sensitivity. Groups with one of these properties are necessarily Liouville. Homomorphisms to free abelian groups provide an obstruction to $\ell^1$-noise sensitivity. We also provide examples of $\ell^1$ and entropy noise sensitive random walks.

Noise sensitivity raises many open questions which are described at the end of the paper.
\end{abstract}

\section{Introduction}

Physically, noise is a non-significant perturbation of an observation. In signal theory, noise is an unintentional perturbation of a message. Noise is an inherent phenomenon to physical observations and to communication. Its influence on a channel capacity was already taken into account by Shannon in his mathematical theory of communication~\cite{Shannon}.

In probability theory, the noise of an event $E(x_1,\dots,x_n)$ (i.e. a Boolean function) depending on a large number of variables can be modeled as the effect of replacing a (small) proportion $\r\in (0,1)$ of the variables by random entries. An event is \emph{noise sensitive} if the realisation of $E$ gives no (or very little) information on what happens when the entries are perturbed by noise. For instance in bond percolation at critical probability $\frac{1}{2}$, having a left-right crossing on an $n\times n$ square lattice is a noise sensitive event \cite{BKS,GPS10}. On the contrary, (weighted) majority functions are noise stable \cite{BKS}. The effect of noise in percolation and voting games has been widely studied over the last decades. We refer to \cite{BKS,GarbanSteiff,GPS10,Kalai2} and references therein for more on these topics.

In the present paper we investigate sensitivity to noise of random walks on groups. A random walk on a group $G$ is a sequence of products $X_n=s_1\dots s_n$ of independent variables $s_i$ following identical distribution $\mu$. Given such a product $X_n$, we can noise it by resampling independently each increment $s_i$ with a probability $\r \in (0,1)$, this provides a new variable $Y_n^\r$ depending on $X_n$. In the Cayley graph, the original word~$X_n$ can be interpreted as a sequence of instructions for a moving particule and~$Y_n^\r$ as the effective trajectory if instructions are misread with probability~$\r$. 

Broadly speaking, the random walk $(G,\mu)$ is \emph{noise sensitive} if $Y_n^\r$ generally seems independent of $X_n$. This vague statement can be specified in a number of ways and we refer to Section~\ref{sec:def} for several precise definitions. 
We retain two principal notions of noise sensitivity in this context. The random walk $(G,\mu)$ is $\ell^1$-noise sensitive if the law of the pair $(X_n,Y_n^\r)$ and the law of a pair $(X_n,X_n')$ of two \emph{independent} samples are close in the sense that their $\ell^1$-distance tends to zero.  The random walk is entropy noise sensitive if the ratio between the conditional entropy $\H(Y_n^\r|X_n)$ and $\H(X_n)$ tends to one, which means that the average amount of information needed to describe $Y_n^\r$ once we already know $X_n$ is (asymptotically) as big as the average amount of information needed to describe $X_n$.

These notions of noise sensitivity are relevant for infinite groups. They are trivially satisfied for finite groups by eventual equidistribution -- see Theorem~\ref{prop:finite}. For simplicity, we restricted our investigations to finitely generated groups and finitely supported probability measures, though the notions make sense in a wider setting.

We first point out two elementary obstructions to $\ell^1$-noise sensitivity.

\begin{theorem} Let $G$ be a countable group and $\mu$ a finitely supported probability measure on $G$.
\begin{enumerate}
\item If $G$ admits a non-trivial homomorphism onto a free abelian group, then $(G,\mu)$ is not $\ell^1$-noise sensitive.
\item If $(G,\mu)$ is $\ell^1$-noise sensitive, then it is Liouville.
\end{enumerate}
\end{theorem} 

This is a concatenation of Corollary \ref{cor:surject} and Theorem \ref{thm:ell1Liouville}. The first part essentially follows from the central limit theorem. For the second part, recall that the space of bounded harmonic functions on $(G,\mu)$ is parametrized by the Poisson boundary \cite{Furstenberg1971}. The random walk $(G,\mu)$ is Liouville if this boundary is reduced to a point, i.e. if all bounded harmonic functions are constant. In the non-Liouville case, the first increment $s_1$ already carries information on the position of $X_n$. For instance in a free group, the first letter of a minimal representative word of $X_n$ is correlated to the first increment. This prevents $\ell^1$-noise sensitivity of the free group. In fact, this obstruction can be strengthen to show that the free group is not even noise sensitive at large scale, see Proposition~\ref{notlargefree}.

By~\cite{KV}, Liouville property is also equivalent to the vanishing of the asymptotic entropy $\lim \frac{1}{n}\H(X_n)$ which measures the average amount of information in one increment of $X_n$. This provides an elementary obstruction to entropy noise sensitivity.

\begin{theorem}
If $(G,\mu)$ is entropy noise sensitive, then it is Liouville.
\end{theorem}

This is a particular case of Theorem \ref{prop:enstoLiouville}. It follows directly from the upper semi-continuity of the asymptotic entropy \cite[Proposition 4]{AAV}. By \cite{FHTVF}, any non-virtually nilpotent finitely generated group admits a symmetric probability measure of finite entropy (usually not finitely supported) for which it is non-Liouville. It is a fortiori neither $\ell^1$-noise sensitive, nor entropy noise sensitive. 

Despite these obstructions, we provide examples of noise sensitive random walks on groups.

\begin{theorem}\label{main3} Regarding entropy noise sensitivity:
\begin{itemize}
\item[$\bullet$] finitely generated abelian groups are entropy noise sensitive with respect to any finitely supported measure,
\item[$\bullet$] the lamplighter group $\Z/2\Z \wr \Z$ is entropy noise sensitive with respect to the switch-walk-switch measure,
\item[$\bullet$] the permutational extension of the first Grigorchuk group described in \cite{BartholdiErschler2012} is partially entropy noise sensitive for some switch and walk measure, i.e.
\[
\exists c >0, \forall \r\in (0,1), \quad \liminf \frac{\H(Y_n^\r|X_n)}{\H(X_n)} \ge c.
\]
\end{itemize}
\end{theorem}

The first point is in Proposition \ref{prop:abelian}, the second is Theorem \ref{thm:lamplighter} and the third point a particular case of Theorem~\ref{thm:dirgps}.

\begin{theorem} On the infinite dihedral group $D_\infty=\langle a,b|a^2,b^2\rangle$,
\begin{itemize}
\item[$\bullet$] the simple random walk is not $\ell^1$-noise sensitive,
\item[$\bullet$] the lazy simple random walk is $\ell^1$-noise sensitive.
\end{itemize}
\end{theorem}

This is Theorem \ref{dihedral}. It follows that $\ell^1$-noise sensitivity is not a geometric property, since it depends on the measure. Moreover as our notions of noise sensitivity are stable under taking direct products (of groups and measures), we obtain noise sensitive random walks on groups commensurable to $\Z^d$ for any rank~$d$.

We expect that morphisms to $\Z$ and non-Liouville property (significance of the first increment) are the only two obstructions for a group to be noise sensitive. This raises the:

\begin{question}\label{conj:ell1}
If $(G,\mu)$ is Liouville and has no virtual morphism onto $\Z$, is it  $\ell^1$-noise sensitive?
\end{question}
Another (unexpected) type of obstruction to $\ell^1$-noise sensitivity would probably be very interesting and significant.
As a particular case, we conjecture that:

\begin{conjecture}\label{conj:Grig}
The first Grigorchuk group is $\ell^1$-noise sensitive.
\end{conjecture}
As for entropy noise sensitivity, we believe it is widely spread and we conjecture:

\begin{conjecture}\label{conj:entropy}
A random walk $(G,\mu)$ is entropy noise sensitive if and only if it is Liouville.
\end{conjecture}
In particular, we believe that $\ell^1$-noise sensitivity implies entropy noise sensitivity -- see Section~\ref{sec:relationship}.

Physically, an $\ell^1$-noise sensitive process can somewhat not be observed, since the observation $Y_n^\r$ does not provide any significant information on the actual output $X_n$. Speculatively, this could account for the rarity of Liouville (probability) groups in natural science. Indeed besides virtually nilpotent ones, all known Liouville groups are genuinely mathematical objects, introduced by mathematicians to solve their problems, without reference to physical sciences \cite{Aleshin1972,KV, Grigorchuk85,Grigorchuk1986,AAV,Brieussel2013,MB,Nekrashevych}.

{\bf Organization of the paper.} Precise definitions of noise sensitivity are given in Section \ref{sec:def} and some of their relationships are studied in Section \ref{sec:relationship}. The effects of Liouville property on noise sensitivity are discussed in Section \ref{sec:Liouville}. Section \ref{sec:fab} is devoted to finite and abelian groups, Section \ref{sec:dih} to the infinite dihedral group. Wreath products are studied in Section \ref{sec:wr}. Some perspectives and open questions are presented in the final Section \ref{sec:qu}.

\section{Notions of noise sensitivity}\label{sec:def}

Let $G$ be a countable group. We endow it with a probability measure $\mu$ whose support generates $G$ as a semi-group. The random walk $(G,\mu)$ is the sequence of random variables $X_n=s_1\dots s_n$ where $(s_k)_{k\ge 1}$ are independent of law $\mu$. The law of $X_n$ is the $n$-fold convolution $\mu_n:=\mu^{\ast n}$.

Given $X_n$ and a \emph{noise parameter} $\r\in[0,1]$, we consider another random variable where the increments $s_k$ are \emph{refreshed} (i.e. resampled independently according to $\mu$) with probability $\r$. More precisely, we define $Y_n^\r=r_1\dots r_n$ where
\[
r_k=\left\{\begin{array}{ll} s_k & \textrm{with probability }1-\r, \\ s'_k & \textrm{with probability }\r, \end{array} \right.
\]
where $s_k'$ is an independent random variable of law $\mu$. We interpret $Y_n^\r$ as a version of $X_n$ perturbed by some noise. Of course the law of $Y_n^\r$ alone is the same as that of $X_n$. 

We denote by $\pi_n^\r$ the joint law of $(X_n,Y_n^\r)$ on $G\times G$. It is the distribution at time $n$ of the random walk on $G\times G$ with measure $\pi^\r:=(1-\r)\mu^{\mathrm{diag}}+\r \mu^2$ where $\mu^{\mathrm{diag}}$ is the diagonal measure on the product taking values $\mu^{\mathrm{diag}}(x,y)=\mu(x)$ when $x=y$ and $\mu^{\mathrm{diag}}(x,y)=0$ otherwise, and $\mu^2:=\mu\times \mu$. In particular, $\pi_n^\r$ is symmetric in the two variables, $\pi_n^1=\mu^2_n$ is the law of a $\mu^2$-random walk on $G\times  G$ at time $n$ and $\pi_n^0=\mu^{\mathrm{diag}}_n$ is the original random walk embedded diagonally in $G \times G$.

In this paper, we investigate the notion of \emph{noise sensitivity} of a random walk, that is how much $Y_n^\r$ for $\r\in(0,1)$ can differ from $X_n$. Informally, we say that $Y_n^\r$ is \emph{noise sensitive} if the couple  $(X_n,Y_n^\r)$  resembles a couple $(X_n,X_n')$ of two independent samples of the random walk, i.e. if the two probability measures $\pi_n^\r$ and $\mu_n^2$ are close. We present several precise quantitative ways of measuring noise sensitivity.

\subsection{Measure-theoretic notions of noise sensitivity} 
Let us start with the most natural notion.

\begin{definition}{\bf $\ell^1$-noise sensitivity.}
The random walk $(G,\mu)$ is \emph{$\ell^1$-noise sensitive} if 
\[
\forall \r\in(0,1),\quad \|\pi_n^\r-\mu_n^2\|_{1}\underset{n\to \infty}{\longrightarrow} 0.
\]
\end{definition}

We recall that the $\ell^1$-distance (or twice total variation) between two probability measures $\xi_1,\xi_2$ on a countable space $E$ is
\[
\|\xi_1-\xi_2\|_{1}:=\sum_{x\in E} |\xi_1(x)-\xi_2(x)| \in [0,2].
\]
The $\ell^1$-distance is also characterised in terms of coupling:
\[
\|\xi_1-\xi_2\|_{1}=\inf_{\n \in \Coup(\xi_1,\xi_2)} \n(x \neq y),
\]
where a coupling $\n$ of $\xi_1$ and $\xi_2$ is a probability measure on $E\times E$ whose marginals are $\xi_1$ and~$\xi_2$. 

Let us now define a weaker notion of noise sensitivity, related to entropy.
Recall that the Shannon entropy of a random variable $X$ of law $\xi$ taking values in a countable set $E$ is 
\[
\H(X)=\H(\xi):=-\sum_{x\in E} \xi(x)\log(\xi(x))=\E_{\xi}I_\xi(x),
\]
where $I_\xi(x)=-\log(\xi(x))$ is the information function. Informally, the entropy of $X$ is the average number of digits needed to describe the value of $X$. Moreover the conditional entropy of another random variable $Y$ with respect to $X$ is
\[
\H(Y|X)=\E_X \H_X(Y)=\sum_{x \in X(E)} \P(X=x)\left( -\sum_{y \in E}\P(Y=y|X=x)\log\P(Y=y|X=x)\right)
\]
i.e. the expectation with respect to $X$ of the entropy of the law of $Y$ conditioned by the value of $X$. Informally this is the average amount of information needed to describe $Y$ when we already know~$X$.

\begin{definition}{\bf Entropy noise sensitivity.}
The random walk $(G,\mu)$ is \emph{entropy noise sensitive}~if
\[
\forall \r\in(0,1), \quad \frac{\H(Y_n^\r|X_n)}{\H(X_n)}\underset{n\to \infty}{\longrightarrow} 1.
\]
The random walk is \emph{partially entropy noise sensitive} if
\[
\exists c>0, \forall \r\in(0,1), \exists N, \forall n \ge N, \quad \H(Y_n^\r|X_n) \ge c\H(X_n).
\]
\end{definition}

Informally speaking, partial entropy noise sensitivity ensures that no matter how small the noise parameter, there is always a fixed proportion of information that cannot be recovered from the noised sample.

It is well known that for any two random variables $\H\left((X,Y)\right)=\H(Y|X)+\H(X)$. In the present context, we get
$\H(Y_n^\r|X_n)=\H(\pi_n^\r)-\H(\mu_n)$.

\subsection{Metric notions of noise sensitivity}\label{sec:met}

As the notion of $\ell^1$-convergence of measures is strong, it can be interesting to relax it. We propose here other notions of noise sensitivity, related to metrics on the group, and motivated by the usual interest in geometric group theory for the large scale aspects of the metric. They can be omitted in first reading. 

Given a metric $d$ on a space $E$, we may define other notions of convergence of measures, for instance via the Wasserstein distances, defined for $p \in [1,\infty)$ by:
\[
W_p(\xi_1,\xi_2):=\inf_{\n \in \Coup(\xi_1,\xi_2)} \E_\n [d(x,y)^p]^{\frac{1}{p}}.
\]
In practice, we will rather use a related notion based on the following quantity: given $s>0$ set
\[
U^s(\xi_1,\xi_2):=\inf_{\n \in \Coup(\xi_1,\xi_2)}\n\left(d(x,y)\ge s\right).
\]
By Markov inequality, we have:
\begin{align}\label{MarkovU}
U^s(\xi_1,\xi_2)^{\frac{1}{p}}\leq \frac{W_p(\xi_1,\xi_2)}{s}
\end{align}

\begin{definition}{\bf Noise sensitivity at scale $s_n$.}
Given a left-invariant distance on $G\times G$ and a sequence $(s_n)$, the random walk $(G,\mu)$ is \emph{noise sensitive at scale $s_n$} if 
\[
\forall \r \in (0,1), \quad U^{s_n}(\pi_n^\r,\mu_n^2) \underset{n\to \infty}{\longrightarrow} 0.
\]
\end{definition}
From a mass transport point of view, this means that for  $n$ large, all but an $\e$ proportion of the ``sand pile" mass distribution of $\mu^2_n$ can be obtained from $\pi^\r_n$, moving sand by distance less than $s_n$.

Unless mentionned otherwise, we assume the distance on $G\times G$ is a word distance, e.g. $d_{G\times G}((x,y),(x',y'))=d_G(x,x')+d_G(y,y')$ where $d_G$ is a word distance on $G$.
We recall that all word distances on a finitely generated group are equivalent up to multiplicative constants.

We point out:
\begin{itemize}
\item Noise sensitivity at scale $1$ is equivalent to $\ell^1$-noise sensitivity when the distance takes integer values, e.g. in discrete groups with word metric.
\item If a Wasserstein distance satisfies $W_p(\pi_n^\r,\mu^2_n) = o(s_n)$, then $(G,\mu)$ is noise sensitive at scale~$s_n$ by (\ref{MarkovU}).
\item If $s'_n\ge s_n$ and $(G,\mu)$ is noise sensitive at scale $s_n$, then it is noise sensitive at scale $s'_n$, because $U^{s'}(\xi_1,\xi_2)\le U^s(\xi_1,\xi_2)$ whenever $s'\ge s$.
\item The quantity $U^s$ resembles a distance in the sense that
\begin{align}\label{Us}
U^s(\xi_1,\xi_3) \le U^{s_1}(\xi_1,\xi_2)+U^{s_2}(\xi_2,\xi_3) \quad \textrm{whenever } s_1+s_2=s.
\end{align}
\end{itemize}
The last point is proved along the same lines that $W_1$ is a distance. Namely let $\nu_{12}$ and $\nu_{23}$ be appropriate couplings to get $U^{s_1}(\xi_1,\xi_2)$ and $U^{s_2}(\xi_2,\xi_3)$. There exists $\g$ a probability on $E^3$ whose projections satisfy $p_{12}\g=\n_{12}$ and $p_{23}\g=\n_{23}$. The coupling $\n_{13}:=p_{13}\gamma$ gives (\ref{Us}).

\begin{definition}{\bf Noise sensitivity at large scale}
The random walk $(G,\mu)$ is {\it noise sensitive at large scale} if there exists a sequence such that $s_n=o(\E d_G(X_n,X'_n))$ and $(G,\mu)$ is noise sensitive at scale $s_n$.
\end{definition}

The {\it spread} $\E d_G(X_n,X'_n)$ of the random walk can be considered as the natural scale to describe the distribution $\mu_n$. Noise sensitivity at large scale essentially means that most refreshed samples $Y_n^\r$ look independent of $X_n$ at the spread scale. For instance, the central limit theorem for virtually abelian groups provides a gaussian description of the distribution of $\mu_n$ at the spread scale $\sqrt{n}$. For such a group, noise sensitivity at large scale implies that the limit gaussian distribution of $(X_n,Y_n^\r)$ is the same as that of $(X_n,X_n')$.

\begin{definition}{\bf  Noise sensitivity in average distance.} Given a left-invariant distance $d_G$ on $G$  (e.g. a word distance when $G$ is finitely generated), the random walk $(G,\mu)$ is \emph{noise sensitive in average distance} if 
\[
\forall \r\in(0,1), \quad \liminf_{n\to \infty}\frac{\E d_G(X_n,Y_n^\r)}{\E d_G(X_n,X'_n)} \ge 1.
\]
The random walk $(G,\mu)$ is \emph{partially noise sensitive in average distance } if \[ \exists c>0,\forall \r\in(0,1), \exists N, \forall n \ge N, \quad \E d_G(X_n,Y_n^\r) \ge c \E d_G(X_n,X'_n).\]
\end{definition} 

\section{Relationships between notions of noise sensitivity}\label{sec:relationship}

The $\ell^1$ noise sensitivity seems to be the strongest among the notions defined above. It implies noise sensitivity at any (in particular at large) scale. 

\subsection{$\ell^1$-noise sensitivity and entropy noise sensitivity}
We believe that $\ell^1$-noise sensitivity implies entropy noise sensitivity, which would be a consequence of Conjecture~\ref{conj:entropy}. We prove it under an additional assumption of homogeneity.

\begin{definition}\label{def:ell1toentropy}
A sequence $(\xi_n)$ of probability measures on $G$ has \emph{homogeneous entropy} if
\[
\mathrm{h}_n(\e):=\sup\left\{\frac{1}{\H(\xi_n)}\sum_{x\in A}\xi_n(x)\log(\xi_n(x))\middle| A\subset G, \xi_n(A) \leq \e \right\} \underset{\e \to 0}{\longrightarrow} 0 
\] 
uniformly in $n$.
\end{definition}
Informally, this means that the smallest atoms do not contribute much to the entropy. 

\begin{proposition}\label{prop:ell1toentropy}
If a random walk $(G,\mu)$ is $\ell^1$-noise sensitive and if the sequence of measures $(\pi_n^\r)_n$ has homogeneous entropy for each $\r \in (0,1]$, then $(G,\mu)$ is entropy noise sensitive.
\end{proposition}

\begin{proof} The case of finite groups is trivial by Proposition \ref{prop:finite}. We assume $G$ is infinite.
Let $\nu_n$ be a coupling between two sequences $(\xi_n^1)$ and $(\xi_n^2)$ with homogeneous entropy such that $\nu_n(x_1\neq x_2) \to 0$. Then for small $\e>0$ and each $i=1,2$, denote $A_\e^i:=\{(1-\e)\xi_n^i(x)\le \nu_n(x,x) \le \xi_n^i(x) \}$. For $n$ large enough, $\xi_n^i(A_\e^i)\ge 1-\e$, thus $\nu_n\{(x,x):x\in A_{\e}^i\}\ge (1-\e)^2\ge 1-2\e$ and by taking marginal $\xi^{3-i}_n(A_\e^i)\ge 1-2\e$. It follows that $\xi^i_n(A_\e^1\cap A_\e^2) \ge 1-3\e$.

As the modulus of continuity of the function $x\log(x)$ is $\omega(\e)=\e |\log \e|$, one has for all $x,y \in (0,1]$
\begin{align*}
\left| \frac{x}{y}-1\right| \le \e \iff |x-y| \le \e y &\implies |x \log x - y \log y| \le \omega(\e y) \le \e y \left| \log \e + \log y \right| \\
& \implies 
 \left| \frac{x \log x}{y \log y}-1 \right| \le \frac{\e |\log \e|}{|\log y|}+\e
\end{align*}
Using the homogeneity assumption:
\begin{align*}
\H(\xi_n^i) &= -\sum_{x} \xi_n^i(x)\log(\xi_n^i(x))=-\sum_{x\in A_\e^1\cap A_\e^2} \xi_n^i(x)\log(\xi_n^i(x)) +\mathrm{h}^i_n(3\e)\H(\xi_n^i) \\
&= - \sum_{x\in A_\e^1\cap A_\e^2} \n_n(x,x)\log(\nu_n(x,x)) \frac{\xi_n^i(x)\log(\xi_n^i(x))}{\n_n(x,x)\log(\nu_n(x,x))} +\mathrm{h}^i_n(3\e)\H(\xi_n^i) \\
&= -(1+\psi^i_n(\e))\sum_{x\in A_\e^1\cap A_\e^2} \n_n(x,x)\log(\nu_n(x,x)) +\mathrm{h}^i_n(3\e)\H(\xi_n^i),
\end{align*}
where $|\psi_n^i(\e)| \le \frac{\e |\log \e|}{\log 2}+\e \underset{\e \to 0}{\longrightarrow}0$ provided all the atoms of $\n_n$ have mass at most $\frac{1}{2}$. We apply these equalities for $\xi_n^1=\pi_n^\r$ and $\xi_n^2=\mu_n^2=\pi_n^1$ (the condition on the atoms of $\n_n$ is satisfied for $n$ large because $G$ is infinite). We get that $\sum_{x\in A_\e^1\cap A_\e^2} \n_n(x,x)\log(\nu_n(x,x))=\H(\pi_n^\r)+o(\H(\pi_n^\r))=\H(\mu_n^2)+o(\H(\mu_n^2))$.
It follows that $\frac{\H(\pi_n^\r)}{\H(\mu_n^2)}\to 1$, which is equivalent to entropy noise sensitivity.
\end{proof}

\begin{remark} For random walks on groups, homogeneous entropy, as well as homogeneous spread (see Definition~\ref{def:ell1todist} below), are related to the tail decay of the distribution $\mu_n$. They are not easy to check on a given group, for it is usually hard to compute entropy or spread in the first place. However they are satisfied for random walks on virtually abelian groups by gaussian decay. They are also satisfied by symmetric random walks on the lamplighter group $\Z/2\Z \wr \Z$, because of the gaussian decay of the range of a random walk on $\Z$ -- see formula (\ref{eq:entropyrange}) in the proof of Theorem~\ref{thm:lamplighter}. Such argument on the range can probably be adapted to the case of iterated lamplighter groups, and possibly to the case of diagonal products of lamplighter groups as in \cite{BZ}. However detailed computations are out of the scope of the present article.
Non-Liouville random walks also have homogeneous entropy and spread by Shannon's theorems, see e.g. \cite{KV}.
We do not know if there exist random walks on groups with non-homogeneous entropy or non-homogeneous spread.
\end{remark}

\subsection{Noise sensitivity at large scale and noise sensitivity in average distance}

We also believe that if $(G,\mu)$ is noise sensitive at large scale, then it is noise sensitive in average (word) distance. Remark~\ref{rem:F} shows that the converse is not true. Again, we prove this under an additional assumption.

\begin{definition}\label{def:ell1todist}
We say the random walk $(G,\mu)$ has \emph{homogeneous spread} if
\[
\mathrm{f}_n(\e):=\sup\left\{\frac{1}{\E d_G(X_n,X'_n)}\sum_{(x,y)\in A}d_G(x,y)\mu^2_n(x,y)\middle| A\subset G\times G, \mu^2_n(A) \leq \e \right\} \underset{\e \to 0}{\longrightarrow} 0 
\]
uniformly in $n$.
\end{definition}
Informally, this means that far away points do not contribute to the spread. 

\begin{proposition}\label{prop:ell1todist}
If a random walk $(G,\mu)$ is noise sensitive at large scale and has homogeneous spread, then it is noise sensitive in average word distance.
\end{proposition}

\begin{proof}
There exists a sequence $s_n=o(\E d_G(X_n,X'_n))$ such that $U^{s_n}(\pi^\r_n,\mu^2_n) \to 0$. So there exists a probability $\nu_n\in \Coup(\pi^\r_n,\mu^2_n)$ such that $\n_n(A_n) \to 1$ where $A_n$ is the event $d_{G^2}((x,y),(x',y'))\le s_n$. We compute
\begin{align*}
\E d_G(X_n,Y^\r_n) &=\sum_{G^2}\pi^\r_n(x,y)d_G(x,y)=\sum_{G^4} \n_n(x,y,x',y')d_G(x,y) \\ 
& \ge \sum_{A_n} \n_n(x,y,x',y')d_G(x',y') -2s_n \quad \textrm{ by triangle inequality,}\\
& = \sum_{G^2} \n_n\left( (G^2 \times \{(x',y')\} )\cap A_n\right)d_G(x',y') -2s_n.
\end{align*}
Given $\e>0$, for $n$ large enough there exist $B_n \subset G^2$ such that $\mu^2_n(B_n) \ge 1-\e$ and 
\[
\forall (x',y') \in B_n, \quad \n_n\left((G^2\times\{(x',y')\})\cap A_n \right) \geq (1-\e)\mu^2_n(x',y').
\]
Then
\begin{align*}
\E d_G(X_n,Y^\r_n) & \ge (1-\e) \sum_{B_n} \mu^2_n(x',y')d_G(x',y') -2s_n \\
&\ge (1-\e)(1-\mathrm{f}_n(\e))\E_{\mu^2_n} d_G(x',y') -2s_n.
\end{align*}
Noise sensitivity in average distance now follows from the homogeneity assumption.
\end{proof}

\section{Noise sensitivity and Liouville property}\label{sec:Liouville}

A function on  a group $G$ is $\mu$-harmonic if $f(y)=\sum_{x \in G} \mu(x)f(yx)$ for all $y$.
A random walk $(G,\mu)$ is Liouville if there are no non-constant bounded $\mu$-harmonic functions. This is equivalent to the fact that the Poisson boundary is reduced to a point \cite{Furstenberg1971} and also equivalent to the fact that the entropy $\H(X_n)$ of the random walk is sublinear \cite{KV}.

\begin{theorem}\label{thm:ell1Liouville}
If $(G,\mu)$ is $\ell^1$-noise sensitive, then $(G,\mu)$ is Liouville.
\end{theorem}
The converse is not true because of infinite abelian groups, see Proposition \ref{prop:abelian}.

\begin{proof}
First observe that $\ell^1$-noise sensitivity implies
\begin{equation}\label{eq:firstdigit}
\|\mu^{\mathrm{diag}}*\mu^2_{n-1}-\mu^2_{n}\|_1 \underset{n\to \infty}{\longrightarrow} 0.
\end{equation}
Indeed, 
taking the convex decomposition of the first factor $\pi^\r=(1-\r)\mu^{\diag}+\r\mu^2$, the $\ell^1$-noise sensitivity first implies $(1-\r)\|\mu^{\diag}\ast \pi^\r_{n-1}-\mu^2_n\|_1 \to 0$. 
On the other hand $\|\mu^{\diag}\ast \pi^\r_{n-1}-\mu^{\diag}\ast\mu^2_{n-1}\|_1\to 0$ by left multiplication. Convergence (\ref{eq:firstdigit}) follows.

Assume by contradiction that $(G,\mu)$ is not Liouville, i.e. admits a non-trivial Poisson boundary $(\Pi,\nu)$. We claim that there exists a subset $A \subset \Pi$ such that
\begin{equation}\label{eq:defA}
\nu(A)^2 < \sum_{z \in G} \mu(z) \nu(z^{-1}A)^2.
\end{equation}
Indeed, as $\nu$ is stationary $\nu(A)=\sum_{z \in G} \mu(z) \nu(z^{-1}A)$ and the above holds with a large inequality by Jensen. As the square function is strictly convex, equality implies $\nu(z^{-1}A)=\nu(A)$ for all $z$ in the support of $\mu$ and a fortiori for all $z$. Therefore equality for all $A$ would imply triviality of the Poisson boundary, whence the claim. 

Now by stationarity of $\nu^2$ with respect to $\mu^2$, we have for each $n$ that
\[
\nu(A)^2=\sum_{(x,y) \in G \times G} \mu_n^2(x,y) \nu(x^{-1}A)\nu(y^{-1}A)
\]
and
\[
\sum_{z \in G} \mu(z) \nu(z^{-1}A)^2=\sum_{(x,y) \in G \times G} \mu^{\mathrm{diag}}*\mu^2_{n-1}(x,y) \nu(x^{-1}A)\nu(y^{-1}A)
\]
Convergence (\ref{eq:firstdigit}) implies that the difference between right hand sides tends to zero,
raising a contradiction to (\ref{eq:defA}) that $\nu(A)^2=\sum_{z \in G} \mu(z) \nu(z^{-1}A)^2$.
\end{proof}

For the free groups, this theorem can be improved.

\begin{theorem}\label{notlargefree}
Finitely supported random walks on free groups are not noise sensitive at large scale.
\end{theorem}

\begin{proof} We follow the same line of reasoning as in the previous proof. We are now given a sublinear sequence $s_n$ such that $U^{s_n}(\pi^\r_n,\mu^2_n)\to 0$. By diagonal coupling of the first factor, we deduce $U^{s_{n-1}}(\pi^\r_n,\pi^\r\ast \mu^2_{n-1})\to 0$. By (\ref{Us}) in Section~\ref{sec:met}, we obtain:
\begin{equation}\label{Us2}
U^{s_n+s_{n-1}}(\pi^\r\ast\mu^2_{n-1},\mu^2_n)\longrightarrow 0.
\end{equation}

Let us consider a subset $A$ satisfying (\ref{eq:defA}) in the Poisson boundary which is now the geometric boundary of the Cayley tree, i.e. the set of infinite geodesic rays out of the neutral element $e$. Up to taking a close enough approximation, we may assume that $A$ is open-closed, i.e. a finite union of ``cylinder" sets. Recall that  ``cylinder" sets $C_h$ are indexed by $h \in \mathbb F$, and $C_h$ is the collection of geodesic rays out of $e$ crossing $h$. Moreover, the point $h$ partitions $\mathbb F \setminus\{h\}$ into the set $\hat C_h$ consisting of points $g$ from which there is a geodesic ray towards $C_h$ not going through $h$ and its complement $\hat C_n^c$.

Recall that the speed (or drift) $\l:=\lim \frac{1}{n}\E d(e,X_n)>0$ is positive. Given $\a \in (0,\l)$, we define $[x]_n$ to be the point at distance $\a n$ from $e$ on the geodesic ray from $e$ to $x$, when $x$ is at distance at least $\a n$ from the identity. Otherwise $[x]_n=e$. We will use the fact that
\begin{equation}\label{conv1}
\P(d(e,X_n) \ge \a n) \underset{n \to \infty}{\longrightarrow} 1.
\end{equation}
We claim that
\begin{equation}\label{acut}
\left|\sum_{G^2} \mu^2_n(g)g\n(A)- \sum_{G^2} \mu^2_n(g)[g]_n\n(A)\right| \underset{n\to\infty}{\longrightarrow} 0
\end{equation}
where we write $g\n(A):=\n(x^{-1}A)\n(y^{-1}A)$ for $g=(x,y)$ and $[g]_n=([x]_n,[y]_n)$. Convoluting by $\pi^\r$ which has finite support, we get
\begin{equation}\label{acut2}
\left|\sum_{G^2}\pi^\r\ast \mu^2_{n-1}(g)g\n(A)- \sum_{G^2} \pi^\r\ast \mu^2_{n-1}(g)[g]_n\n(A)\right| \underset{n\to\infty}{\longrightarrow} 0.
\end{equation}
 Now let $\n_n$ denote appropriate couplings to get (\ref{Us2}), we have:
\begin{align*}
&\left|\sum_{G^2} \mu^2_n(g) [g]_n\n(A)-\sum_{G^2} \pi^\r\ast\mu^2_{n-1}(g)[g]_n\n(A)\right| \\ &\quad =\left|\sum_{(g_1,g_2)\in G^4} \n_n(g_1,g_2)\left([g_1]_n\n(A)-[g_2]_n\n(A)\right)\right| \\
&\quad \le 2\sum_{(g_1,g_2) \in G^4} \n_n(g_1,g_2)\mathbf{1}_{\{[g_1]_n\neq [g_2]_n\}}\\
&\quad \le U^{s_n+s_{n-1}}(\pi^\r\ast \mu^2_{n-1},\mu^2_n)+4 \P\left[d_G(e,X_n)\le \a n+s_n+s_{n-1}\right] \underset{n\to\infty}{\longrightarrow} 0
\end{align*}
where the last inequality is due to the geometry of the tree that if both $x,x'$ are at distance $\ge \a n+s$ from identity and $d(x,x')\le s$, then $[x]_n=[x']_n$. The convergence to $0$ follows by (\ref{Us2}) and (\ref{conv1}) for $\a <\a'<\l$ as $ \a n+s_n+s_{n-1}<\a'n$ for $n$ large. Together with (\ref{acut}) and (\ref{acut2}), we deduce
\[
\left|\sum_{G^2} \mu^2_n(g) g\n(A)-\sum_{G^2} \pi^\r\ast\mu^2_{n-1}(g)g\n(A)\right| \underset{n\to\infty}{\longrightarrow} 0
\]
raising the same contradiction as at the end of the proof of Theorem~\ref{thm:ell1Liouville}.

There remains to prove the claim (\ref{acut}). As $A$ is a finite union of cylinder sets, it is enough to prove that for any $h_1,h_2$ in $\mathbb F$, one has
\begin{equation}\label{42f}
\left| \sum_{(x,y) \in G^2}\mu_n(x)x\nu(C_{h_1})\mu_n(y)y\nu(C_{h_2})- \sum_{(x,y) \in G^2}\mu_n(x)[x]_n\nu(C_{h_1})\mu_n(y)[y]_n\nu(C_{h_2}) \right| \underset{n\to\infty}{\longrightarrow} 0
\end{equation}
We use the facts that $x\nu(A)=\nu(x^{-1}A)$ is the probability that a random walk started in $x$ tends to a boundary point in $A$, and that for any $\e>0$ there exists $N$ such that when $d(x,z) \ge N$, the probability that a random walk started in $x$ ever hits $z$ is at most $\e$. Using union bounds, we can deduce that for any $R \ge 0$, there exists $N'$ such that when $d(x,h_1) \ge N'$, a random walk started at $x$ hits the ball $B(h_1,R)$ with probability at most $\e$.
We use it for $R$ large enough that a random walk trajectory cannot move from $\hat C_{h_1}$ to its complement (or vice-versa) without stepping in $B(h_1,R)$. It follows that for $d(e,x)\ge \a n \ge N'+d(e,h_1)$ we have
\begin{align*}
&\textrm{ if }x \in \hat C_{h_1}, \textrm{ then } x\nu(C_{h_1}) \ge 1-\e \textrm{ and } [x]_n\nu(C_{h_1}) \ge 1-\e, \\
&\textrm{ if }x \in \hat C_{h_1}^c, \textrm{ then } x\nu(C_{h_1}) \le \e \textrm{ and } [x]_n\nu(C_{h_1}) \le \e.
\end{align*}
Finally, let us denote $\Lambda_n$ the set of points $(x,y)$ in $G^2$ such that $d(e,x)\ge \a n$ and $d(e,y)\ge \a n$ where $n$ is large enough to have $d(e,x)\ge \a n \ge N'+d(e,h_1)$ and the analogous property for $y,h_2$ in place of $x,h_1$.  For $(x,y) \in \Lambda_n$, one has 
\[
|x\nu(C_{h_1})y\nu(C_{h_2})- [x]_n\nu(C_{h_1})[y]_n\nu(C_{h_2})| \le 5\e
\]
by checking the four cases.
Then the expression in (\ref{42f}) is 
\[
\left| \sum_{(x,y) \in G^2}\mu_n(x)\mu_n(y) \left(x\nu(C_{h_1})y\nu(C_{h_2})- [x]_n\nu(C_{h_1})[y]_n\nu(C_{h_2})\right)\right| \le 2 \mu_n^2(\Lambda_n^c)+5\e \mu_n^2(\Lambda_n) \le 4\e+5\e
\] where $ \mu_n^2(\Lambda_n^c)$ is bounded above using (\ref{conv1}) for $n$ large.
\end{proof}

The above proof can probably be generalised to groups with hyperbolic properties. However in full generality, it is not clear if sublinearly close points have close actions on the harmonic measure.

\begin{theorem}\label{prop:enstoLiouville}
If $(G,\mu)$ is partially entropy noise sensitive, then $(G,\mu)$ is Liouville.
\end{theorem}

This is one direction of Conjecture \ref{conj:entropy}.

\begin{proof}
Assume $(G,\mu)$ is not Liouville, which means that the associated Poisson boundary is not trivial. By \cite{KV}, this implies
\[
h_\infty(\mu):=\lim_{n\to \infty} \frac{\H(\mu_n)}{n}>0.
\]
Thus $\H(X_n)=\H(\mu_n)=h_\infty(\mu)n +o(n)$. Similarly
\[
h_\infty(\pi^\r):=\lim_{n\to \infty} \frac{\H(\pi_n^\r)}{n} \ge h_\infty(\mu)
\]
because $\H(\pi_n^\r)=\H(X_n,Y_n^\r)=\H(Y_n^\r|X_n)+\H(X_n) \ge \H(X_n)=\H(\mu_n)$. It follows that $\H(\pi_n^\r)=h_\infty(\pi^\r)n+o(n)$.

Moreover $\pi^\r=(1-\r)\mu^{\diag}+\r \mu^2$ tends to $\mu^\diag$ weakly as $\r \to 0$, and $\H(\pi^\r) \to \H(\mu^\diag)=\H(\mu)$ as these measures have support in the finite set $\supp(\mu)^2$. The upper semi-continuity of asymptotic entropy \cite[Proposition 4]{AAV} implies 
\[
\lim_{\r \to 0} h_\infty(\pi^\r)=h_\infty(\mu).
\]
Therefore we have
\[
\lim_{n \to \infty}\frac{\H(Y_n^\r|X_n)}{\H(X_n)}=\frac{h_\infty(\pi^\r)-h_\infty(\mu)}{h_\infty(\mu)} \underset{\r \to 0}{\longrightarrow} 0,
\]
ruling out partial entropy noise sensitivity.
\end{proof}

\begin{remark}\label{rem:F}
Noise sensitivity in average distance does not imply Liouville property. Indeed for the simple random walk on the free group $\mathbb{F}_2$, one easily computes $\E d_{\mathbb{F}_2}(X_n,X'_n) \sim n \sim \E d_{\mathbb{F}_2}(X_n,Y_n^\r)$ for any $\r>0$.
It follows that average distance noise sensitivity does not imply large scale noise sensitivity.
\end{remark}

\section{Finite and abelian groups}\label{sec:fab}

\subsection{Finite groups}

\begin{proposition}\label{prop:finite}
Any finite group is $\ell^1$-noise sensitive, entropy noise sensitive and  noise sensitive in average distance with respect to any generating probability measure. 
\end{proposition}

\begin{proof}
The measure $\pi^\r:=(1-\r)\mu_{\mathrm{diag}}+\r \mu^2$ is generating of $G^2$, so $\pi_n^\r$ satisfies $\|\pi_n^\r-\mathrm{unif}_{G^2}\|_1 \le \b_1^n$ for some $\b_1<1$ (see e.g. \cite{Saloff-Coste-RWFG}). Similarly $\|\mu_n^2-\mathrm{unif}_{G^2}\|_1 \le \b_2^n$ for some $\b_2<1$. Therefore the random walk in $G$ is $\ell^1$-noise sensitive, with total variation converging exponentially fast to zero.

Since the set $G$ is finite, it follows that $\H(\pi_n^\r)\rightarrow \H(\mathrm{unif}_{G^2})=2\H(\mathrm{unif}_G)$ and $\H(\mu_n) \rightarrow \H(\mathrm{unif}_G)$. This implies entropy noise sensitivity. 

Moreover let $Z$ denote a uniformly random variable in $G$, it is immediate that $|\E d_G(X_n,X'_n)-\E d_G(e,Z)|\le \diam(G)\|\mu_n-\unif_G\|_1$ and 
\[
|\E d_G(X_n,Y_n^\r)-\E d_G(e,Z)|\le \diam(G)\left(\|\mu_n-\unif_G\|_1+\|\pi_n^\r-\mathrm{unif}_{G^2}\|_1\right).
\] 
By $\ell^1$-noise sensitivity, both converge to zero, whence distance noise sensitivity.
\end{proof}

This result is not surprising since we defined noise sensitivity only asymptotically. An interesting further question is whether it is possible that noise sensitivity manifests itself before cut-off, that is before the random walk seems equidistributed. However we will not pursue in this direction.

\subsection{Infinite abelian groups}

In abelian groups, computations are explicit, based on the local central limit theorem.

\begin{proposition}\label{prop:abelian}
Let $G$ be an infinite finitely generated abelian group and $\mu$ be any finitely supported probability measure. Then 
\begin{itemize}
\item $(G,\mu)$ is entropy noise sensitive,
\item $(G,\mu)$ is not, even partially,  noise sensitive in average word distance,
\item $(G,\mu)$ is not $\ell^1$-noise sensitive.
\end{itemize}
\end{proposition}

The proof of the first point is based on the following folklore lemma.

\begin{lemma}\label{lemma:H}
Let $X_n$ be a random walk on $\Z^d$ with step distribution $\mu$ finitely supported and generating. Then
\[
\H(X_n)=\frac{d}{2}\log n+o(\log n).
\]
\end{lemma}

\begin{proof}[Proof of Lemma \ref{lemma:H}]
We denote $m=\E X_1$.
We fix $\e>0$. The central limit theorem gives $K$ such that $\P(|X_n-nm| \le K \sqrt{n})\ge 1-\e$. Then
\begin{equation}\label{Hsplit}
\H(X_n)=\H(X_n||X_n-nm|\le K\sqrt{n})\P(|X_n-nm|\le K\sqrt{n})+\H(X_n||X_n-nm|> K\sqrt{n})\P(|X_n-nm|> K\sqrt{n})+h_1,
\end{equation}
where $0\le h_1 \le \log 2$ is the entropy of the conditioning partition.
The local central limit theorem \cite{DMD,lawler_limic_2010} provides two constants $c_1,c_2>0$ depending only on $K$ such that for all n large enough, for all $|x-nm| \le K\sqrt{n}$,
\begin{equation}\label{lclt}
\frac{c_1}{n^{\frac{d}{2}}} \le \P(X_n=x||X_n-nm| \le K \sqrt{n}) \le \frac{c_2}{n^{\frac{d}{2}}}
\end{equation}
It follows that for $n$ large enough
\begin{equation}\label{Hbd}
\frac{d}{2}\log(n)+\log\frac{1-\e}{c_2} =  \log \frac{(1-\e)n^{\frac{d}{2}}}{c_2} \le \H(X_n||X_n-nm|\le K\sqrt{n}) \le \frac{d}{2}\log(n)+\log\frac{1-\e}{c_1}
\end{equation}
 and the lower bound together with (\ref{Hsplit}) give $\liminf \frac{\H(X_n)}{\log n} \ge \frac{d}{2}$.
 
 On the other hand, as $X_n$ is confined in a ball of radius $Cn$ and of volume at most $C'n^d$, we have
 \[
 \H(X_n||X_n-nm|> K\sqrt{n}) \le \log(C'n^d)=d \log(n)+\log(C').
 \]
 For $n$ large enough, we deduce with (\ref{Hsplit}) and (\ref{Hbd})
 \[
 \H(X_n) \le (1-\e)\frac{d}{2}\log(n)+\e d\log(n)+C''
 \]
 which gives $\limsup \frac{\H(X_n)}{\log n} \le \frac{d}{2}$.
\end{proof}

\begin{proof}[Proof of Proposition \ref{prop:abelian}] Such a group $G$ is isomorphic to $\Z^d \times F$ for some finite abelian group $F$. Lemma~\ref{lemma:H} still holds in this more general setting for the finite group $F$ contributes at most a constant to entropy.

We first prove entropy noise sensitivity, we have to show $\lim \frac{\H(Y_n^\r|X_n)}{\log n}=\frac{d}{2}$. The inequality $\H(Y_n^\r|X_n) \le \H(X_n)$ provides the upper bound. 
 
 Let $\ell$ be the number of refreshed variables in $Y_n^\r$. Given $\ell$, we can write $(X_n,Y_n^\r)=(X_{n-\ell}X_\ell,X_{n-\ell}X'_\ell)$ where $X_\ell$ and $X_\ell'$ are independent, of law $\mu_\ell$. They are also independent of $X_{n-\ell}$ which has law $\mu_{n-\ell}$. As $\ell$ follows a binomial law $\mathcal{B}(n,\r)$, we have for any $\e>0$
 \begin{equation}\label{roeps}
 \P(A_\e)=\P(\ell \in [(\r-\e)n,(\r+\e)n]) \underset{n\to \infty}{\longrightarrow} 1
 \end{equation}
 where we denote $A_\e=\{\ell \in [(\r-\e)n,(\r+\e)n]\}$. As further conditioning does not increase entropy, we have the lower bound
 \begin{equation*}
 \H(Y_n^\r|X_n)  = \H(X_{n-\ell}X_\ell'|X_n)  \ge \H(X_{n-\ell}X_\ell'|X_n, X_{n-\ell},\ell)=\H(X_\ell'|X_n,X_{n-\ell},\ell)=\H(X_\ell'|\ell).
 \end{equation*}
It follows that
\begin{equation*}
\H(Y_n^\r|X_n) \ge \sum_{k=0}^{n} \H(X_k')\P(\ell=k) \ge \sum_{k=(\r-\e)n}^{n} \H(X_k')\P(\ell=k) \ge \H(X'_{(\r-\e)n})\P(\ell \ge (\r-\e)n) 
\end{equation*}
and thus by Lemma \ref{lemma:H} and (\ref{roeps})
\begin{equation*}
\H(Y_n^\r|X_n) \ge \left(\frac{d}{2}\log((\r-\e)n)+o(\log n)\right)(1+o(1))=\frac{d}{2}\log n+o(\log n)
\end{equation*}
which is the required lower bound to get entropy noise sensitivity.

Regarding the spread, there exist two constants $c_3,c_4>0$ depending only on $\mu$ such that
\begin{equation}\label{spreadab}
c_3\sqrt{n} \le \E d_G(X_n,X_n') \le c_4 \sqrt{n}.
\end{equation}
Conditioning as above by the number $\ell$ of refreshed variables, we have
\[
\E d_G(X_n, Y_n^\r)=\E d_G(X_{n-\ell}X_\ell,X_{n-\ell}X_\ell')=\E d_G(X_\ell,X_\ell')=\sum_{k=0}^{Cn} \E \left( d_G(X_\ell,X_\ell')|\ell=k\right)\P(\ell=k)
\]
where $C=2\max\{d_G(e,x)|x \in \supp \mu\}$.
By splitting the sum at $(\r+\e)n$ for arbitrary $\e>0$ and using (\ref{roeps}), we get
\[
\E d_G(X_n, Y_n^\r) \le c_4\sqrt{(\r+\e)n}\P(\ell \le (\r+\e)n)+c_4\sqrt{Cn}\P(\ell >(\r+\e)n)=c_4\sqrt{\r+\e}\sqrt{n}(1+o(1)).
\]
With (\ref{spreadab}) we deduce that
\[
\limsup_{n \to \infty} \frac{\E d_G(X_n,Y_n^\r)}{\E d_G(X_n,X_n')}\le \frac{c_4}{c_3} \sqrt{\r} \underset{\r \to 0}{\longrightarrow} 0.
\]
This rules out partial noise sensitivity in  average word distance, and thus $\ell^1$-noise sensitivity by Proposition \ref{prop:ell1todist}.
\end{proof}

\begin{corollary}\label{cor:surject}
Let $(G,\mu)$ be a group with a surjective homomorphism $G\twoheadrightarrow \Z$ and $\mu$ finitely supported, then $(G,\mu)$ is not $\ell^1$-noise sensitive.
\end{corollary}

\begin{proof} It follows from Proposition~\ref{prop:abelian}
because $\ell^1$-noise sensitivity is preserved under taking quotients $G\twoheadrightarrow \bar{G}$ (for the induced measure): if $\mu_1,\mu_2$ are two measures on $G$, then $\|\bar{\mu_1}-\bar{\mu_2}\|_1 \le \|\mu_1-\mu_2\|_1$. 
\end{proof}

\subsection{Product groups}

\begin{proposition}\label{prop:product}
Let $G_1,G_2$ be two finitely generated groups with probabilities $\mu_1$ and $\mu_2$ respectively. If both $(G_1,\mu_1)$ and $(G_2,\mu_2)$ are noise sensitive in the sense of one definition of section~\ref{sec:def}, then $(G_1\times G_2,\mu_1\times \mu_2)$ is noise sensitive as well in this definition.
\end{proposition}

\begin{proof}
It follows straightforwardly from the definitions.
\end{proof}

The choice of product probability measure is important here, because noise sensitivity notions may depend on the choice of probability. It is the case  for $\ell^1$-noise sensitivity and noise sensitivity in average distance by Corollary \ref{cor:dpm}.

\section{The infinite dihedral group}\label{sec:dih}

Let $D_\infty=\langle a,b|a^2=b^2=1\rangle$ denote the infinite dihedral group. The random walk on $D_\infty$ driven by $\frac{1}{2}(\d_a+\d_b)$ is called simple. The random walk driven by $\frac{1}{3}(\d_e+\d_a+\d_b)$  is called lazy simple.

\begin{theorem}\label{dihedral} Let $D_\infty$ be the infinite dihedral group.
\begin{itemize}
\item The simple random walk on $D_\infty$ is not $\ell^1$-noise sensitive.
\item The lazy simple random walk on $D_\infty$ is $\ell^1$-noise sensitive.
\end{itemize}
\end{theorem}

The Cayley graph of $D_\infty$ with respect to the generating set $\{a,b\}$ is a line, so it coincides with the Cayley graph of $\Z$. Therefore random walks on $D_\infty$ are related to random walks on the integers, which explains the first statement. However the edge labelings are very different. This difference will be key to the second statement.

\begin{corollary}\label{cor:dns}
The simple random walk on $D_\infty$ is not  noise sensitive in average word distance. The lazy simple random walk on $D_\infty$ is  noise sensitive in average word distance.
\end{corollary}

\begin{proof}[Proof of Corollary \ref{cor:dns}]
The first statement is justified below in the proof of Theorem \ref{dihedral}. The second statement follows from Proposition~\ref{prop:ell1todist}.
\end{proof}

Note that both simple and lazy simple random walks on $D_\infty$ are entropy noise sensitive by the same argument as in abelian case.

\begin{corollary}\label{cor:dpm}
The $\ell^1$-noise sensitivity and the  noise sensitivity in average word distance are not group properties but depend on the probability measure. A fortiori they are not preserved under quasi-isometries.
\end{corollary}

\begin{corollary}
For any positive integer $d$, there exists an $\ell^1$-noise sensitive probability group commensurable with $\Z^d$.
\end{corollary}

\begin{proof}
Take a direct product $D_\infty^d$ and use Proposition \ref{prop:product}.
\end{proof}
 
\begin{remark}
Theorem \ref{dihedral} shows that virtually abelian groups may be $\ell^1$ noise sensitive. Informally, this is due to the noise sensitivity of the action of the finite quotient on the torsion free subgroup. It is obvious that if we are given a trajectory in the streets of New York by a sequence of moves North-South-East-West and we misread one instruction, we will still end up close to the aim. However if the instructions are given in terms of Forward-Backward-Left-Right and we miss a turn, we will most likely end up very far from the aim. It would be interesting to understand precisely when a virtually abelian group is noise sensitive or not.
\end{remark}

\subsection{Couplings of (lazy or not) simple random walks on $\Z$}

We record here three lemmas that will be used in the proof of Theorem~\ref{dihedral}.  The first one is very classical and the other two are easy consequences.

\begin{lemma}\label{coups}
Let $S_n$, resp. $S_n'$, denote a (lazy or not) simple random walk on $\Z$ started at $0$, resp. at $x\sqrt{n}$. There exists $\s \in \{0,1\}$ such that
\[
\limsup_{n\to \infty}||\mathrm{law}(S_n)-\mathrm{law}(S'_n-\s)||_1 \underset{x\to 0}{\longrightarrow}0.
\]
For lazy simple random walks, this holds true with $\s=0$.
\end{lemma}

\begin{proof} Let us start with the case of (non lazy) simple random walk.
We assume $x\sqrt{n}$ is even and take $\s=0$. The case $x\sqrt{n}$ odd and $\s=1$ is similar. We consider the following coupling of $S_n$ and $S'_n$ : while $S_k\neq S'_k$, sample the increments independently, once $S_k=S'_k$ take the same increments. This ensures that $S_k=S'_k$ for any $k \ge T=\min\{k: S_k=S'_k\}$. We have $||\rm{law}(S_n)-\rm{law}(S'_n-\s)||_1\le \P(T>n)$. 

Now while $k\le T$, the random variable $S_k'-S_k$ has the law of a random walk with step distribution $\frac{1}{2}\d_0+\frac{1}{4}\d_2+\frac{1}{4}\d_{-2}$. Then $\P(T>n)$ is the probability that such a random walk, starting at the even integer $x\sqrt{n}$, does not hit $0$ by time $n$. The later tends to $0$ with $x$ by \cite[Theorem~2.13]{Rev}.

For lazy simple random walks, the random variable $S_k'-S_k$ has a step distribution charging positively each increment in $\{-2,-1,0,1,2\}$, and it can hit $0$ as above without parity issue.
\end{proof}

We will actually rather use the proof than the lemma, and we call the coupling above the standard coupling between two simple random walks on $\Z$ started at different positions.

\begin{lemma}\label{2Scoup}
Let $S_n$, resp. $S_n'$, denote a (lazy or not) simple random walk on $\Z$ started at $0$, resp. at $x\sqrt{n}$, and let $\delta\in \mathbb R$. There exists $\s \in \{0,1\}$ such that
\[
\limsup_{n\to \infty}||\rm{law}(S_n)-\rm{law}(S'_{(1+\delta)n}-\s)||_1 \longrightarrow 0
\]
when both $x$ and $\delta$ tend to $0$. This holds true with $\s=0$ for lazy simple random walks.
\end{lemma}

\begin{proof}
This is a consequence of the previous lemma. We consider only the (non lazy) simple random walk and assume that $\d>0$ and that both $x\sqrt{n}$ and $\d n$ are even. Similar arguments apply in the other cases. We first sample $S'_{\d n}$. By the central limit theorem, for any $\e>0$, there exists $A>0$ such that for $n$ large
\[
\P\left[|S'_{\d n}|\le A\sqrt{\d n}\right] \ge 1-\e.
\]
Conditioning on this event, we apply Lemma \ref{coups} to $S_n$ and $S_{\d n}'^{-1}S'_{(1+\d)n}$ which is starting at $x'\sqrt{n}$ with $x-A\d \le x' \le x+A\d$. The probability to achieve their standard coupling is greater than $1-\e$ as long as $x$ and $\d$ are small enough. It follows that the probability to couple $S_n$ and $S'_{(1+\delta)n}$ is at least $1-2\e$ for $n$ large.
\end{proof}

In order to ease notations, we write $[a\pm b]$ as a shortcut for the interval $[a-b,a+b]$.

\begin{lemma}\label{3Scoup}
Assume $\ell^+,\ell^- \in \left[\left(\frac{4}{9}\pm\d\right)n\right]$ and $\ell \in \left[(1\pm\d)n\right]$, then for any given integer $s$, 
\[
\limsup_{n\to \infty}\left\Vert \left(\frac{1}{2}\d_1+\frac{1}{2}\d_{-1} \right)^{\ast \ell^+}\ast \left(\frac{1}{2}\d_0+\frac{1}{4}\d_1+\frac{1}{4}\d_{-1} \right)^{\ast \ell^-}- \d_s\ast\left(\frac{1}{3}\d_0+\frac{1}{3}\d_{1}+\frac{1}{3}\d_{-1}\right)^{\ast\ell}\right\Vert_1 \underset{\d\to0}{\longrightarrow} 0.
\]
\end{lemma}
We could replace $s$ by $x\sqrt{n}$ with $x$ tending to $0$ but this will not be necessary. There is no shift $\s\neq 0$ because of the lazy aspect of the random walk.

\begin{proof} Fix $\e>0$. Set $\mu_{\mathrm{S}}=\frac{1}{2}\d_1+\frac{1}{2}\d_{-1}$ and $\mu_{\mathrm{LS}}=\frac{1}{3}\d_0+\frac{1}{3}\d_1+\frac{1}{3}\d_{-1}$. Then
\[
\left(\frac{1}{2}\d_1+\frac{1}{2}\d_{-1} \right)^{\ast \ell^+}\ast \left(\frac{1}{2}\d_0+\frac{1}{4}\d_1+\frac{1}{4}\d_{-1} \right)^{\ast \ell^-}=\mu_{\mathrm{S}}^{\ast \ell^+}\ast \left( \frac{3}{4}\mu_{\mathrm{LS}}\right)^{\ast \ell^-}=\mu_{\mathrm{S}}^{\ast \ell^+}\ast \left( \mu_{\mathrm{LS}}\right)^{\ast m}
\]
with $1-\e \le \P\left[ m \in \left[ \left(\frac{3}{4}\pm \d\right)\ell^- \right] \right]\le \P\left[ m \in \left[ \left(\frac{1}{3}\pm 2\d\right)n \right] \right]$ for $n$ large. Similarly we have
\[
\d_s\ast \mu_{\mathrm{LS}}^{\ast \ell}=\d_s\ast\mu_{\mathrm{LS}}^{\ast\frac{2}{3}n}\ast \mu_{\mathrm{LS}}^{\ast \ell-\frac{2}{3}n}=\d_s\ast\left(\frac{2}{3}\mu_{\mathrm{S}}\right)^{\ast \frac{2}{3}n} \ast \mu_{\mathrm{LS}}^{\ast \ell-\frac{2}{3}n}=\d_s\ast\mu_{\mathrm{S}}^{\ast m'} \ast \mu_{\mathrm{LS}}^{\ast \ell-\frac{2}{3}n}
\]
with $1-\e \le \P\left[ m' \in \left[ \left(\frac{2}{3}\pm \d\right)\frac{2}{3} n\right] \right]\le \P\left[ m' \in \left[ \left(\frac{4}{9}\pm \d\right)n \right] \right]$ for $n$ large. Under these generic conditions on $m$ and $m'$, Lemma~\ref{2Scoup} gives a probability no less than $1-\e$ that the two random walks with respective laws $\mu_{\mathrm{S}}^{\ast \ell^+}$ and $\d_s\ast\mu_{\mathrm{S}}^{\ast m'}$ achieve a standard coupling (up to adding $1$), provided $\d$ is small enough, and $n$ large. Conditioning by such a success, Lemma~\ref{2Scoup}, applied to random walks of respective laws $\left( \mu_{\mathrm{LS}}\right)^{\ast m}$ and $\mu_{\mathrm{LS}}^{\ast \ell-\frac{2}{3}n}$ started at most $1$ apart, gives a probability  no less than $1-\e$ that the two considered random walks achieve coupling. We conclude that
\[
\left\Vert \left(\frac{1}{2}\d_1+\frac{1}{2}\d_{-1} \right)^{\ast \ell^+}\ast \left(\frac{1}{2}\d_0+\frac{1}{4}\d_1+\frac{1}{4}\d_{-1} \right)^{\ast \ell^-}- \d_s\ast\left(\frac{1}{3}\d_0+\frac{1}{3}\d_{1}+\frac{1}{3}\d_{-1}\right)^{\ast\ell}\right\Vert_1\le 4\e
\]
provided $\d$ is small enough and $n$ large.
\end{proof}

\subsection{Proof of Theorem~\ref{dihedral}}

 Let us now describe precisely the labelling of the Cayley graphs of $\Z$ and $D_\infty$, pictured in Figure~\ref{picture}.
In the integers, from each vertex there is an edge to the right labelled by $+1$ and and edge to the left labelled by $-1$. In the dihedral group, let us say that the words $(ab)^k$ correspond to even positions and the words $(ab)^ka$ to odd positions. At each even position, there is an edge to the left labelled by $a$ and an edge to the right labelled by $b$. At each odd position, this is the converse.

\begin{figure}
\begin{centering}
\begin{tikzpicture}

\node[draw,circle] (-2)at(-6,0) {$\begin{array}{c}ba\\ -2\end{array}$};
\node[draw,circle] (-1)at(-3,0) {$\begin{array}{c}b\\ -1\end{array}$};
\node[draw,circle] (0)at(0,0) {$\begin{array}{c}e\\ 0\end{array}$};
\node[draw,circle] (1)at(3,0) {$\begin{array}{c}a\\ 1\end{array}$};
\node[draw,circle] (2)at(6,0) {$\begin{array}{c}ab\\ 2\end{array}$};

\draw[->,>=latex, thick, blue] (0) to[bend left] (1);
\draw[blue] (1.5,1.1) node{$\begin{array}{c}a\\ +1\end{array}$};
\draw[->,>=latex, thick, blue] (1) to[bend left] (0);
\draw[blue] (1.5,-1.1) node{$\begin{array}{c}a\\ -1\end{array}$};

\draw[->,>=latex, thick, blue] (-2) to[bend left] (-1);
\draw[blue] (-4.5,1.1) node{$\begin{array}{c}a\\ +1\end{array}$};
\draw[->,>=latex, thick, blue] (-1) to[bend left] (-2);
\draw[blue] (-4.5,-1.1) node{$\begin{array}{c}a\\ -1\end{array}$};

\draw[->,>=latex, thick, red] (-1) to[bend left] (0);
\draw[red] (-1.5,1.1) node{$\begin{array}{c}b\\ +1\end{array}$};
\draw[->,>=latex, thick, red] (0) to[bend left] (-1);
\draw[red] (-1.5,-1.1) node{$\begin{array}{c}b\\ -1\end{array}$};

\draw[->,>=latex, thick, red] (1) to[bend left] (2);
\draw[red] (4.5,1.1) node{$\begin{array}{c}b\\ +1\end{array}$};
\draw[->,>=latex, thick, red] (2) to[bend left] (1);
\draw[red] (4.5,-1.1) node{$\begin{array}{c}b\\ -1\end{array}$};

\draw[blue] (7,0) node{$\cdots$};
\draw[red] (-7,0) node{$\cdots$};
\end{tikzpicture}

\par\end{centering}

\caption{\label{picture} Cayley graph of $D_\infty$ and identification with $\Z$.}
\end{figure}

\begin{proof}[Proof of the first part of Theorem \ref{dihedral}]
For the simple random walk, the key observation is that position $X_n$ and time $n$ have the same parity. The path parametrized by $X_n$ corresponds to a random walk on the integers where the increments at odd times are $+1$ for $a$ and $-1$ for $b$, and the increments at even times are $+1$ for $b$ and $-1$ for $a$. This provides a simple random walk on the integers, which is neither $\ell^1$-noise sensitive nor noise sensitive in average word distance by Proposition~\ref{prop:abelian}.
\end{proof}

For lazy random walks, time and position no longer have the same parity.
Let us first observe the effect of refreshing one increment. Denote $X_n=v_0sv_1$ and $Y_n=v_0rv_1$ with $s,r$ independent, $\mu$-distributed in $\{a,b,e\}$. By abuse of notation, we denote $X_n=v_0+s+v_1$ the corresponding path in the integers, as above. 

When $s\in \{a,b\}$ and $r=e$ (or vice-versa), it corresponds in the integers to $s=\pm 1$ (the sign depends on the parity of $v_0$) being replaced by $r=0$ (or vice-versa). Thus $v_0s$ and $v_0r$ have different parity, so the next moves of the random walk, described by the word $v_1$, will be mirrored of the moves in the original walk. We write $Y_n=v_0+r-v_1$.

When $s=a$ and $r=b$ (or vice-versa), then $s=\pm 1$ (depending on $v_0$) is replaced by $r=\mp 1$. The parity is not modified and we write $Y_n=v_0+r+v_1$.

Now denote $X_n=v_0s_{j_1} v_1 s_{j_2} \dots s_{j_m}v_{m}$ a lazy simple random walk where $s_{j_i}$ are the increments to be refreshed and $v_i=s_{j_i+1}\dots s_{j_{i+1}-1}$ are words. With the notations above, it corresponds to a sum
\begin{equation}\label{eq:Xdih}
X_n=v_0+s_{j_1}+ v_1+s_{j_2}+v_2 \dots +s_{j_m}+v_{m}
\end{equation}
in the integers. The refreshed sample $Y_n^\r=v_0r_{j_1}v_1\dots r_{j_m}v_{m}$ corresponds to a sum
\begin{equation}\label{eq:Ydih}
Y_n^\r=\a_0v_0+r_{j_1}+\a_1 v_1+r_{j_2}+\a_2v_2 \dots +r_{j_m}+\a_{m} v_{m}
\end{equation}
where $\a_i =\pm1$ are given by $\a_0=1$ and $\a_{i}=-\a_{i-1}$ if and only if $s_i \in \{a,b\}$ and $r_i=e$ or vice-versa, i.e. there is a change of parity, which occurs with probability $4/9$. This is summarised in Table~\ref{table}.

\begin{figure}
\begin{centering}
\begin{tabular}{| l | l |c|c|c|c|c|c|c|c|c|}
\hline
  & & $L_0$ & \multicolumn{4}{c|}{$\a_i=-\a_{i-1}$} & \multicolumn{4}{c|}{$\a_i=\a_{i-1}$} \\
\hline 
 & $s_i\in D_\infty$ & $e$ & $a$  & $b$ & $e$ & $e$ & $a$ & $a$ & $b$ & $b$ \\
 \cline{2-11}
  & $r_i\in D_\infty$ & $e$ &$e$  & $e$ & $a$ & $b$ & $a$ & $b$ & $a$ & $b$ \\
   \hline
 even position  & $s_i\in \Z$ & $0$ & $1$  & $-1$ & $0$  & $0$  & $1$ & $1$ & $-1$ & $-1$ \\
    \cline{2-11}
    & $r_i\in \Z$ & $0$  & $0$   & $0$  & $1$ & $-1$ & $1$ & $-1$ & $1$ & $-1$ \\
     \hline
 odd position   & $s_i\in \Z$ & $0$  & $-1$  & $1$ & $0$  & $0$  & $-1$ & $-1$ & $1$ & $1$ \\
    \cline{2-11}
      & $r_i\in \Z$ & $0$  &$0$   & $0$  & $-1$ & $1$ & $-1$ & $1$ & $-1$ & $1$ \\
\hline 
\end{tabular}
\par\end{centering}

\caption{\label{table}Correspondence of the effect of noising in the dihedral and integer model. The position is the integer $v_0+ \dots +s_{j_{i-1}}+v_{i-1}$, as in (\ref{eq:Xdih}).}
\end{figure}

At this stage, we can get an intuition of the result because the sum $\sum r_j$ is actually independent of the sum $\sum s_j$ and the Lyapunov central limit theorem applied to the random variables $\pm v_i$ ensures that the sums $\sum v_i$ and $\sum \a_i v_i$ are essentially independent. Making this precise would prove noise sensitivity at large scale. 

However, to obtain $\ell^1$ convergence, we need to construct a coupling between $(X_n,Y_n^\r)$ and $(X_n,X'_n)$. For this, we will use the additive model of expressions (\ref{eq:Xdih}) and (\ref{eq:Ydih}). 
\begin{proof}[Proof of the second part of Theorem \ref{dihedral}]
Let us first explain how we sample $(X_n,Y_n^\r)$ in the integer model (\ref{eq:Xdih},\ref{eq:Ydih}). We first sample the locations $L$ of the refreshed increments, according to a Bernoulli law of parameter $\r$. Among $L$, we sample the locations $L_0$ where $(s_i,r_i)=(e,e)$ (or $(s_i,r_i)=(0,0)$ in the integer model), according to a Bernoulli law of parameter $\frac{1}{9}$. Then we sample the $v_i$, i.e. the increments $s_i$ for $i$ not in $L$ according to $\frac{1}{3}(\d_e+\d_a+\d_b)$. Then we sample the values of $\a_i$ according to a Bernoulli law of parameter $\frac{1}{2}$. This partitions $L\setminus L_0$ into $L_\a^+ \sqcup L_\a^-$. Finally, we sample the pairs $(s_i,r_i)$ according to $\frac{1}{4}(\d_{(1,-1)}+ \d_{(1,1)}+\d_{(-1,1)}+\d_{(-1,-1)})$  for $i$ in $L_\a^+$ and according to 
$\frac{1}{4}(\d_{(1,0)}+ \d_{(-1,0)}+\d_{(0,1)}+\d_{(0,-1)})$ for $i$ in $L_\a^-$. In view of Table~\ref{table}, this yields a sample of $(X_n,Y_n^\r)$ in the integer model.

Now we explain how to couple $(X_n,Y_n^\r)$ with $(X_n,X'_n)$. More precisely, we will couple the conditioned law $(Y_n^\r|X_n)$ with $X_n'$ (which is equal to $(X_n'|X_n)$ by independence) on a large enough subset.

We first sample the subsets $L$ and $L_0$, thus we have integer expressions of $X_n$ and $Y_n^\r$ as in (\ref{eq:Xdih},\ref{eq:Ydih}). We will mimic these expression for $X'_n$ and write it in the form
\[
X'_n=v_0'+s'_{j_1}+v'_1+s'_{j_2}+v'_2 \dots +s'_{j_m}+v'_{m}+\sum_{i\in L_0} s'_i
\]
where $v'_i=\sum_{i \in \{j_i+1,\dots,j_{{i+1}-1}\}\setminus L_0}s'_i$. In other words, we pull out the terms in $L_0$. Then we will call length of $v_i$, and denote by $\ell(v_i)$, the cardinal of the set $\{j_i+1,\dots,j_{i+1}-1\}\setminus L_0$. This choice makes sure that $\ell(v_i)=\ell(v_i')$ for all $i$. This length follows a geometric law of parameter $\frac{1-\frac{8}{9}\r}{1-\frac{1}{9}\r}$ (taking into account the specificity of $L_0$). It follows that for any integer $k\ge 0$ there is a constant $c_k>0$ such that for any $\d>0$
\begin{align}\label{c_k}
\P\left[\#\{i: \ell(v_i)=k\} \in \left[(c_k\pm\d)n\right]\right] \underset{n\to \infty}{\longrightarrow} 1.
\end{align}

In order to construct a coupling between $(X_n,Y_n^\r)$ and $(X_n,X_n')$, we will take advantage of commutativity of the addition of integers in order to decompose the conditioned process $(Y_n^\r|X_n)$ and $X_n'$ into a finite sequence of independent rescaled simple random walks, run during close times. For this, we fix an integer $K$ to be determined later and define:
\begin{align}\label{Xsl}
X_n^{\rm{long}}=\sum_{\{i :\ell(v_i)>K\}} v_i \quad\textrm{and}\quad  X_n^{\rm{short}}=\sum_{k=0}^K \sum_{\{i :\ell(v_i)\le K,  |v_i|=k\}} v_i =\sum_{k=0}^K k S_{m(k)}^{(k)}
\end{align}
and similarly
\begin{align}\label{Ysl}
Y_n^{\rm{long}}=\sum_{\{i :\ell(v_i)>K\}} \alpha_iv_i   \quad\textrm{and}\quad Y_n^{\rm{short}}=\sum_{k=0}^K \sum_{\{i :\ell(v_i)\le K,  |v_i|=k\}} \a_iv_i= \sum_{k=0}^K k S_{m(k)}''^{(k)} 
\end{align}
where $m(k)$ is the cardinal of the subset $\{i :\ell(v_i)\le K,  |v_i|=k\}$. The sum of the $v_i$'s on the latter subset has the law of a simple random walk, which we denote $S_{m(k)}^{(k)}$, rescaled by the multiplicative factor~$k$. As the $\a_i$'s are independent uniform on $\{\pm1\}$, the sum of the $\a_iv_i$'s indexed by this very subset is an  independent simple random walk $S_{m(k)}''^{(k)}$.
We also denote 
\[
X_n^{\a+}=\sum_{j\in L_\a^+} s_j,  \quad  X_n^{\a-}=\sum_{j\in L_\a^-} 
s_j,  \quad  Y_n^{\a+}=\sum_{j\in L_\a^+} 
r_j,  \quad  Y_n^{\a-}=\sum_{j\in L_\a^-} 
r_j,
\]
in order to have
\[
X_n=X_n^{\rm{long}}+X_n^{\rm{short}}+X_n^{\a+}+X_n^{\a-} \quad \textrm{and} \quad Y_n^\r=Y_n^{\rm{long}}+Y_n^{\rm{short}}+Y_n^{\a+}+Y_n^{\a-}.
\]
Similarly we decompose $X'_n=X_n'^{\rm{long}}+X_n'^{\rm{short}}+X_n'^{\a+}+X_n'^{\a-}+X_n'^{0}$ into
\[
X_n'^{\rm{long}}=\sum_{\{i :\ell(v'_i)>K\}} v'_i, \quad  X_n'^{\rm{short}}=\sum_{k=0}^K \sum_{\{i :\ell(v'_i)\le K,  |v'_i|=k\}} v_i =\sum_{k=0}^K k S_{m'(k)}'^{(k)},
\]
and
\[
X_n'^{\a+}=\sum_{j \in L_\a^+} s_j', \quad X_n'^{\a-}=\sum_{j \in L_\a^-} s_j', \quad X_n'^0=\sum_{j\in L_0}s_j',
\]

Using (\ref{c_k}) and the distributions of the $v_i$'s conditioned by their lengths, we obtain for each $k\ge 0$ a constant $c'_k>0$ such that for any $\d>0$ and $n$ large
\begin{align}\label{m(k)}
\P\left[m(k) \in \left[ (c'_k\pm\d)n \right] \right] \underset{n\to \infty}{\longrightarrow} 1
\end{align}
and the same holds for $m'(k)$. Observe that $c'_1>c_1$. Also note that by (\ref{c_k}) and tail decay of the geometric law, there exists a sequence $\e_K\underset{K\to \infty}{\longrightarrow}0$ such that
\begin{align}\label{ellK}
\P\left[\ell(X_n^{\rm{long}}) \le \e_Kn \right]\underset{n\to \infty}{\longrightarrow} 1
\end{align}
where $\ell(X_n^{\rm{long}})=\sum_{i:\ell(v_i)>K}\ell(v_i)$ is the length of $X_n^{\rm{long}}$ as a word (we removed the trivial increments in $L_0$).

We are now ready to describe the required coupling. We fix some $\e>0$. By Lemma~\ref{2Scoup}, there exists $x,\d>0$ such that any two simple random walks started distance $\le x\sqrt{n}$ apart and run for times in $[(1\pm\d)n]$ achieve a standard coupling with probability at least $1-\e$, provided $n$ is large enough. Moreover by the central limit theorem applied to a (lazy or not) simple random walk $S_n$, there exists $A>0$ such that for $n$ large
\[
\P\left[|S_n| \le A\sqrt{n} \right] \ge 1-\e.
\]
We choose $K$ large enough that $\frac{4A\sqrt{\e_K}}{c_1}\le x$. The central limit theorem and (\ref{ellK}) show that 
\[
\P\left[|X_n^{\rm{long}}|\le A\sqrt{\e_K n} \right]\ge 1-\e
\]
and the same inequality holds with $Y_n^{\rm{long}}$ or $X_n'^{\rm{long}}$ in place of $X_n^{\rm{long}}$. Let $B_0=\{|Y_n^{\rm{long}}|>A\sqrt{\e_Kn}\}\cup\{|X_n'^{\rm{long}}|>A\sqrt{\e_Kn}\}$ denote the ``bad" set at step $0$. We have $\P(B_0)\le 2\e$. 

We also choose $\d_1>0$ such that $\frac{\d_1}{c_1}\le \d$. We set $B_1=\{m(1) \notin [(c_1'\pm\d_1)n]\} \cup \{m'(1) \notin [(c_1'\pm\d_1)n]\}$.  Applying (\ref{m(k)}) with $k=1$, we have $\P(B_1)\le \frac{2\e}{K}$ provided $n$ is large enough.

Conditioned on $X_n$ and on $B_0^\complement \cap B_1^\complement$, the random variables $Y_n^{\rm{long}}+S_{m(1)}''^{(1)}$ and $X_n'^{\rm{long}}+S_{m'(1)}'^{(1)}$ are simple random walks started distance $|Y_n^{\rm{long}}-X_n'^{\rm{long}}|\le 2A\sqrt{\e_K}\sqrt{n}$ apart and run for times $m(1),m'(1) \in [(c_1'\pm\d_1)n]\subset [(1\pm\d)c'_1n]$ by choice of $\d_1$. The choice of $K$ permits to use Lemma~\ref{2Scoup} and we obtain for $n$ large
\begin{align}\label{G_1}
\P\left[|Y_n^{\rm{long}}+S_{m(1)}''^{(1)}-(X_n'^{\rm{long}}+S_{m'(1)}'^{(1)})|\le 1 \left| B_0^\complement \cap B_1^\complement, X_n\right. \right] \ge 1-\e.
\end{align}

Now we will use Lemma~\ref{2Scoup} repeatedly as follows. For $\e'=\frac{\e}{K}>0$, there exists $\d'>0$ such that two simple random walks started distance at most $1$ apart and run for times in $[(1\pm\d')n]$ achieve a standard coupling with probability at least $1-\e'$. For each $2\le k \le K$, we choose $\d_k>0$ with $\frac{\d_k}{c'_k}\le \d'$. We denote $B_k$ the ``bad'' event that $\{m(k),m'(k)\}$ is not included in $[(c'_k\pm\d_k)n]$. For $n$ large enough, we have $\P(B_k) \le 2\e'$ by (\ref{m(k)}), so the ``bad'' event $B=\cup_{k=0}^K B_k$ has probability $\P(B)\le 4\e$ and thus $\P(B^\complement) \ge 1-4\e$.

Let us denote $G_k$ the ``good'' event at step $k$, still conditioned by $X_n$, that 
\[
G_k=\left\{ \left| Y_n^{\rm{long}}+S_{m(1)}''^{(1)}+ \dots +kS_{m(k)}''^{(k)}-\left(X_n'^{\rm{long}}+S_{m'(1)}'^{(1)}+\dots+kS_{m'(k)}'^{(k)}\right) \right|\le 1 \left| X_n \right. \right\}
\]

Inequality (\ref{G_1}) implies $\P(G_1|B^\complement, X_n) \ge 1-\e$.
We also have:
\[
 \forall 2 \le k\le K, \quad \P(G_k|B^\complement, G_{k-1},X_n)\ge 1-\e'.
 \] 
Indeed under the conditions $X_n$, $G_{k-1}$ and $B^\complement$, the event $G_k$ has the probability that two simple random walks started at most $1$ apart and run for times $m(k), m'(k) \in [(c'_k\pm\d_k)n] \subset \left[ (1\pm\d')c'_kn \right]$ achieve a standard coupling, which holds for $n$ large with probability no less than $1-\e'$ by choice of $\d'$. It follows that 
\begin{align}\label{GK}
\P(G_K|B^\complement, X_n) \ge 1-\e-(K-1)\e'\ge 1-2\e.
\end{align}

There remains to consider the contributions of the pairs $(s_i,r_i)$ for $i \in L\setminus L_0$ and of $s_i'$ for $i\in L$. Note that as the values of $\a_i$'s have been sampled (which guaranteed the independence of $S_{m(k)}^{(k)}$ and $S_{m(k)}''^{(k)}$), the locations $L_\a^+$ and $L_\a^-$ are determined already.  For some $\d_\a>0$, let us denote
\[
B_\a=\left\{\#L_\a^+ \notin \left[\left(\frac{4\r}{9}\pm\d_\a\right)n \right]\right\}\cup\left\{\#L_\a^- \notin \left[\left(\frac{4\r}{9}\pm\d_\a\right)n \right]\right\}\cup\left\{\#L_0 \notin \left[\left(\frac{\r}{9}\pm\d_\a\right)n \right]\right\}
\]
the ``bad'' event that the partition $L =L_\a^+\sqcup L_\a^- \sqcup L_0$ is not ``generic''. We have for any $\d>0$ and $n$ large that $\P(B_\a^\complement)\ge 1-\e$.

According to Table~\ref{table}, the laws of $Y_n^\a+$ and $Y_n^{\a-}$ conditioned by $B_\a^\complement,X_n,L_\a^+,L_\a^-$ are $\left(\frac{1}{2}\d_1+\frac{1}{2}\d_{-1}\right)^{\ast \#L_\a^+}$ and $\left(\frac{1}{2}\d_0+\frac{1}{4}\d_{1}+\frac{1}{4}\d_{-1}\right)^{\ast \#L_\a^-}$ respectively. The law of $X_n'^{\a+}+X_n'^{\a-}+X_n'^0$ is $\left(\frac{1}{3}\d_0+\frac{1}{3}\d_{1}+\frac{1}{3}\d_{-1}\right)^{\ast \#L}$. 

Note that $G_K=\left\{ \left| Y_n^{\rm{long}}+Y_n^{\rm{short}}-\left(X_n'^{\rm{long}}+X_n'^{\rm{short}}\right) \right|\le 1 \left| X_n \right. \right\}$. We apply Lemma~\ref{3Scoup} with $|\s|\le 1$, and obtain, provided $\d_\a$ is small enough and $n$ large,
\[
\P\left[Y_n^\r=X_n'\left| G_K,B_\a^\complement,X_n \right. \right]\ge 1-\e.
\]
Combined with (\ref{GK}), we get
\[
\P\left[Y_n^\r=X_n'\left| B^\complement,B_\a^\complement,X_n \right. \right]\ge \P\left[Y_n^\r=X_n'\left| G_K,B^\complement,B_\a^\complement,X_n \right. \right]\P\left[G_K\left| B^\complement,B_\a^\complement,X_n \right. \right] \ge 1-3\e.
\]
Integrating over the condition $X_n$, we get $\P\left[(X_n,Y_n^\r)=(X_n,X_n')\left| B^\complement,B_\a^\complement\right. \right]\ge 1-3\e$. Finally as $\P\left(B^\complement\right)\le 4\e$ and $\P\left(B_\a^\complement\right)\le \e$, we obtain for $n$ large that $\P\left[ (X_n,Y_n^\r)=(X_n,X_n')\right] \ge 1-8\e$
\end{proof}

Observe that in the proof above the parity issue between $X_n$ and $Y_n^\r$ can only be solved using the final ``lazy'' $Y_n^{\a-}$ part, because $X_n^{\rm{long}}$, respectively $X_n^{\rm{short}}$, $X_n^{\a+}$, has the same parity as $Y_n^{\rm{long}}$, respectively $Y_n^{\rm{short}}$, $Y_n^{\a+}$.

\section{Wreath products}\label{sec:wr}

Let $G$ and $\Lambda$ be two groups. Assume $G$ acts transitively on a set $S$. The permutational wreath product of $G$ and $\Lambda$ over $S$ is the semi-direct product
\[
\Lambda \wr_S G:=\left(\bigoplus_S \Lambda\right) \rtimes G.
\]
Its elements are pairs $(f,g)$ where $f:S \to \Lambda$ is finitely supported and $g$ belongs to $G$. The action of $G$ on finitely supported functions is by translations $gf(\cdot)=f(\cdot g)$. Given an arbitrary point $o$ in $S$, a natural generating set is the union of elements $(\id,g)$ for $g$ in some generating set of $G$ together with elements $(\l\delta_o,\id)$ for $\l$ in some generating set of $\Lambda$ where $\l\delta_o(x)=\l$ if $x=o$ and $\l\delta_o(x)=\id$ otherwise.

When $S=G$ is acted upon by the right regular representation, we recover the (usual) wreath product and we simply denote it by $\Lambda \wr G$.

\subsection{The lamplighter group}\label{sec:lamplighter}

The lamplighter group is the wreath product $\Z/2\Z \wr \Z$. Its elements are pairs $(f,t)$ where $t$ is an integer and $f:\Z\rightarrow \Z/2\Z$ is a finitely supported function. As the action is by shift on $\Z$, the product is $(f,t)(f',t')=(f(\cdot)+f'(\cdot+t),t+t')$. One can think of $t$ as a position of a lighter and $f$ as a space of configurations of lamps on or off. The group is generated by $(0,\pm1)$, which correspond to moves of the lighter, and $(1\delta_0,0)$, where $1\delta_0(x)$ takes value $1$ for $x=0$ and value $0$ otherwise, which corresponds to switching on or off the lamp at the lighter's position. Let $\mu_1$ be equidistributed on $\{(0,0),(1\d_0,0)\}$ and $\mu_2$ be equidistributed on $(0,\pm1)$, then the measure $\mu:=\mu_1 \ast \mu_2 \ast \mu_1$ is called the "switch-walk-switch" measure.

\begin{theorem}\label{thm:lamplighter}
The lamplighter group with switch-walk-switch measure is entropy noise sensitive and partially  noise sensitive in average word distance.
\end{theorem}

This group is not $\ell^1$-noise sensitive by Corollary~\ref{cor:surject}. 

\begin{proof}
Let $X_n=s_1\dots s_n=(f_n,x_n)$ denote a sample. The projection $x_n$ to the integers is a simple random walk. We denote 
\[
\Loc(x,n):=\{0\le t \le n \ : \ x_t=x\}
\]
the local time at $x$ (both the set and its cardinal are called local time by a slight abuse). Conditioned on the local time, the state of the lamp at $x$ is given by:
\[
f_n(x)=\prod_{t \in \Loc(x,n)} \a_t, \quad \textrm{where }\a_t \textrm{ are i.i.d. uniform in }\Z/2\Z.
\]
Let us denote $R_n:=\{ x \in \Z : \Loc(x,n)>0\}$ the range of the random walk projected to the integers. The entropy of $X_n$ is given by
\begin{align}\label{eq:entropyrange}
\H(X_n)=\E |R_n|+O(\log n)
\end{align}
because to describe $X_n$ we need to provide the lamp configuration on the range and the position (which has only logarithmic entropy). Note that we use logarithm in base $2$. It is well-known that $\E |R_n|\asymp \sqrt{n}$. Moreover, by gaussian decay, the expected size of range is homogeneous in the sense that for $n$ large, one has $\E\mathbbm{1}_A|R_n|/\E|R_n| \le f(\P(A))$ for some function $f(\e)\underset{\e\to 0}{\longrightarrow}0$.

Now consider a refreshed sample $Y_n^\r=r_1\dots r_n=(g_n,y_n)$. Again conditioning by the trajectory projected onto the integers, it appears that the conditional entropy satisfies 
\begin{align}\label{eq:condrange}
\H(Y_n^\r|X_n) \ge \E|R_n^{\textrm{ref}}(Y_n^\r)|,
\end{align}
 where $R_n^{\textrm{ref}}(Y_n^\r):=\left\{ x \in R_n(Y_n^\r) \ : \ \exists t \in \Loc(x,n), r_t=s'_t \textrm{ was refreshed}\right\}$,
because for $x$ in this set of refreshed lamps, the value $g_n(x)$ is independent of the sample $X_n$. Moreover, for a given $x$ in the range of $Y_n^\r$, the probability that the lamp is not refreshed is precisely $(1-\r)^{\Loc(x,n)}$, so our conditional entropy is related to the distribution of local time. The Ray-Knight theorem guarantees that
\[
\P\left[\#\left\{x \in R_n \ : \ \Loc(x,n)\le \d \sqrt{n} \right\} \le \d \sqrt{n} \right] \underset{\d \to 0}{\longrightarrow} 1.
\]
So for any $\e>0$, there exists an $\d>0$ such that the event
\[
A_\d:=\left\{\#\left\{x \in R_n(Y_n^\r) \ : \ \Loc(x,n)\ge \d\sqrt{n} \right\} \ge (1-\e)\left|R_n(Y_n^\r)\right|\right\}
\] 
has probability $\P(A_\d)\ge 1-\e$.
In this set each lamp is refreshed with probability at least $1-(1-\r)^{\d\sqrt{n}}\ge 1-\e$, provided $n$ is large enough, so:
\[
\P\left[ |R_n^{\textrm{ref}}(Y_n^\r)|\ge \left(1-2\e\right) \#\left\{x \in R_n(Y_n^\r) \ : \ \Loc(x,n)\ge \d\sqrt{n} \right\} \right] \underset{n \to \infty}{\longrightarrow} 1.
\]
It follows that conditioned on $A_\d$ we have for $n$ large:
\[
\P\left[ |R_n^{\textrm{ref}}(Y_n^\r)|\ge \left(1-3\e\right)\left|R_n(Y_n^\r)\right| \left| A_\d \right.\right] \ge 1-\e.
\]
Denote $B$ the event that $ |R_n^{\textrm{ref}}(Y_n^\r)|\ge \left(1-3\e\right)\left|R_n(Y_n^\r)\right| $, we finally have $\P(B) \ge 1-2\e$ for $n$ large. We can compute:
\begin{align*}
\E|R_n^{\textrm{ref}}| = \E\mathbbm{1}_B|R_n^{\textrm{ref}}|+\E\mathbbm{1}_{B^\complement}|R_n^{\textrm{ref}}| \ge (1-3\e)\E\mathbbm{1}_B|R_n| \ge (1-3\e)(1-f(2\e))\E|R_n|.
\end{align*}
Together with (\ref{eq:entropyrange}) and (\ref{eq:condrange}), this implies entropy noise sensitivity.

The distance between two elements $(f,x)$ and $(g,y)$ is the minimal number of steps for the lighter to start from position $x$, switch all lamps at positions $t$ with $f(t)\neq g(t)$ and go to position $y$. In particular, $\E d(X_n,Y_n^\r) \ge \E|R_n^{\textrm{ref}}(Y_n^\r)|-o(\sqrt{n})\asymp \sqrt{n}$. This implies partial noise sensitivity in average word distance.
\end{proof}

\subsection{A lower bound for permutational wreath products} The ideas in the previous proof can be used in arbitrary permutational wreath products, but it is usually difficult to obtain informations about local times. We give a weaker statement which will be used in the next section.

Given an action of $G$ on $S$, recall that the inverted orbit of a point $x$ in $S$ under a word $w=s_1\dots s_n$ is the set 
\[
\O(w)=\left\{x,\ xs_1^{-1}, \ xs_{2}^{-1}s_1^{-1}, \ \dots, \ xs_n^{-1}\dots s_1^{-1}\right\}.
\] 
A switch-walk measure on $\L \wr_S G$ is a measure of the form $\mu_\L \ast \mu_G$ where $\mu_G$ is an arbitrary measure on $G$ and $\mu_\L$ is an arbitrary measure on the copy of $\L$ siting over a fixed point $o$ in $S$, namely $\L=\{\l\d_o : \l \in \L\}$. 

\begin{lemma}\label{lemma:Hinv}
Let $\L$ be a finite group. A switch-walk random walk on $\L \wr_S G$ satisfies
\[
\H(Y_n^\r | X_n) \ge \r \H(\mu_\L)\E\left|\O(X_n)\right|.
\]
\end{lemma}

By slight abuse of notation, we write on the righthand side $\O(X_n)$ where we mean the inverted orbit of the inverse of a sample path (or word) of the random walk and not only its evaluation in the group.
Heuristically this lemma is simply a lower bound on the expected number of lamps refreshed.
In its own, it does not provide information on noise sensitivity because the righthand side depends on the noise parameter $\r$.

\begin{proof}
Let $X_n^{\word}=\l_1\d_{o}s_1\dots \l_n\d_{o}s_n$ denote a sample path of the random walk with evaluation $X_n=(f_n,g_n)$ in the group $\L \wr_S G$. The lamp at $x$ takes value 
\[
f_n(x)=\l_1\d_{o}(x)\l_2\d_{o}(xs_1)\dots \l_n\d_{o}(xs_1\dots s_n)=\prod_{t \in \Loc(x,n)} \l_t,
\]
where $\Loc(x,n):=\left\{0 \le t \le n \ : x.s_1\dots s_{t-1}=o \right\}$ and  $\l_t$ are independent of law $\mu_\L$.
Note that points with positive local time are precisely points on the inverted orbit of $X_n^{\word}$.  

Now to sample $Y_n^\r$, we first sample the locations of refreshed increments, then resample only the $\mu_G$ factors of the increments. This gives an intermediate word $Y_n'^{\word}$ whose evaluation $Y_n'$ in the group individually has the same law as $X_n$. Then we refresh the $\mu_\L$ factors of the increments. We have :
\[
\H(Y_n^\r|X_n) \ge \H(Y_n^\r|X_n^{\word}) \ge \H(Y_n^\r|X_n^{\word}, Y_n'^{\word})=\H(Y_n^\r| Y_n'^{\word}).
\]
The last equality is due to the fact that all the factors in the word $X_n^{\word}$ agree with the factors of $Y_m'^{\word}$ except those that will be resampled according to $\mu_\L$ independently to obtain $Y_n^\r$.

For $x$ in $S$, the lamp at $x$ is refreshed from $Y_n'^{\word}$ to $Y_n^\r$ with probability $1-(1-\r)^{\Loc(x,n)}$, which is $\ge \r$ as soon as $x$ has positive local time, i.e. is in the inverted orbit of $Y_n'^{\word}$. It follows that $\H(Y_n^\r|Y_n'^{\word})\ge \r \H(\mu_\L)\E|\O(Y_n'^{\word})|=\r \H(\mu_\L)\E\left|\O(X_n^{\word})\right|$.
\end{proof}

\subsection{Permutational wreath products of some groups acting on rooted trees}

Let $T=T_d$ be a rooted tree of degree $d$. We view it as the graph with vertex set $\sqcup_{\ell\ge 0} X^\ell$ where $X=\{1,\dots,d\}$ and edges between any two vertices of the form $v$ and $vx$ for $x \in X$. The root is the empty sequence obtained for $\ell=0$, and the set $X^\ell$ is called the $\ell$th level. The tree boundary is the set $\partial T=X^{\Z_+}$. The group $\Aut(T)$ of (rooted) automorphisms of $T$ is canonically isomorphic to its permutational wreath product with the symmetric group $\Sym_X$ over $X$
\begin{equation}\label{autT}
\Aut(T) \simeq \Aut(T) \wr_X \Sym_X.
\end{equation}
This is obtained by identification between $T$ and the subtrees $xT$ rooted at vertices 
$x$ in the first level. We identify an element and its image under this isomorphism to write $g=\langle g|_1,\dots,g|_{d}\rangle \s$ where $g|_x \in \Aut(T)$ is called the section at $x$ and $\s \in \Sym_X$ is the root permutation. The action of $g \in \Aut(T)$ on a vertex of the form $xu$ for $x \in X$ is given by $g.xu=\s.x g|_x u$. The product rule is
$gg'=\langle g|_1g|_{\s^{-1}(1)},\dots,g|_dg|_{\s^{-1}(d)}\rangle\s\s'$. This isomorphism can be iterated to the $\ell$th level as 
\[
\Aut(T) \simeq \Aut(T) \wr_X \Sym_X \wr_X \dots\wr_X\Sym_X \hookrightarrow \Aut(T) \wr_{X^\ell} \Sym_{X^\ell},
\]
and in particular, the section of $g$ at any vertex is defined by induction as $g|_{vx}=(g|_v)|_x$ 
We refer to \cite{GriNewH} or \cite{Nekbook} for details on the structure of $\Aut(T)$ and wreath product isomorphisms.

We consider a finitely generated subgroup $G=G(A,H)$ of $\Aut(T)$ with a generating set of the form $A \cup H$ where $A$ and $H$ are two subgroups such that
\begin{itemize}
\item all the elements of $A$ are rooted, i.e. have all their sections of the first level trivial, and $A$ acts transitively on the first level,
\item all the elements of $H$ have the form $h=\langle h|_1,\dots,h|_d\rangle \s$ with sections $h|_1 \in H$ and $h|_x\in A$ for all $x\in X\setminus\{1\}$ and permutation fixing $1$, i.e. $\s1=1$.
\item let $\u_H$ and $\u_A$ denote respectively the uniform measure on $H$ and $A$, we assume moreover that the first section map $h\mapsto h|_1$ pushes $\u_H$ forward to $\u_H$ itself, and that the other section maps $h\mapsto h|_x$ for $x \in X\setminus\{1\}$ push $\u_H$ forward to $\u_A$ (this together with transitivity of the $A$-action on $X$ implies in particular that the image of $G$ under the section map $g\mapsto g|_v$  is $G$ itself, and this for any vertex $v$).
\end{itemize}

Recall that the first Grigorchuk group is the group acting on a binary rooted tree generated by four elements recursively defined by $a=\langle e,e\rangle (12)$, $b=\langle c,a\rangle$, $c=\langle d,a\rangle$ and $d=\langle b,e\rangle$. It satisfies this set of hypotheses by putting $A=\{e,a\}$ and $H=\{e,b,c,d\}$. Other classical exemples covered by our hypotheses are the various Grigorchuk groups \cite{Grigorchuk85,Grigorchuk1986}, as well as generalisations in \cite{BartholdiSunik2001}, and mother automata groups of degree $0$ introduced in \cite{BKN}.

We now consider some specific permutational wreath products $\Lambda \wr_S G$ of a group $G=G(A,H)$ with a finite group $\Lambda$ over the set $S=G.1^\infty\subset \partial T$, which is the orbit of the boundary point $1^\infty$ under the action of $G$. This group is generated by $A \cup H\cup \Lambda$ where we identify $\Lambda$ with the subgroup $\{\l\delta_{1^\infty}: \l \in \Lambda\}$ where $\l\delta_{1^\infty}(x)$ takes value $\l$ at $x=1^\infty$ and $e$ elsewhere. Under identifications of the tree boundaries $\partial T=\sqcup_{x \in X} \partial xT$, the isomorphism (\ref{autT}) extends to the action on the boundary and we obtain an embedding
\[
\Lambda \wr_S G(A,H) \hookrightarrow \left(\Lambda \wr_S G(A,H) \right)\wr_X A.
\]
Notice that under this embedding $\l\delta_{1^\infty}=\langle \l\delta_{1^\infty},e,\dots,e\rangle$. Moreover it is immediately checked that $H$ and $\Lambda$ commute. For such groups, we have:

\begin{theorem}\label{thm:dirgps}
Let $G=G(A,H)$ be as above and let $\Lambda$ be a finite group. The random walk $(\L \wr_S G, \u_{\Lambda H}\ast \u_A)$  is partially entropy noise sensitive.
\end{theorem}

The measure $\u_{\Lambda H}\ast \u_A=\u_\L \ast \u_{HA}$ was called ``switch and walk'' in the statement of Theorem~\ref{main3}.
This theorem applies for instance to extensions of the Grigorchuk groups introduced by Bartholdi and Erschler, and which can have prescribed growth, see \cite{BartholdiErschler2012, Brieussel2014,BartholdiErschler2014}. For simplicity we have decided to restrict ourselves to the case of tree of constant degree, but this hypothesis is unnecessary and the arguments of the proof below apply to rooted trees with bounded degree. Such generalisations provide interesting groups  with prescribed entropy when $\Lambda$ is finite and prescribed speed when $\Lambda=\Z$, see \cite{Brieussel2013,AmirVirag2012}.

\begin{proof}  For these groups, Amir and Virag \cite[Inequality (9) and Corollary 24]{AmirVirag2012} have shown that the entropy of the random walk satisfies
\begin{equation}\label{AV4}
c \E\left|\O(X_n)\right| \le H(X_n) \le C \E\left|\O(X_n)\right|
\end{equation}
 where $\O(X_n)$ is the inverted orbit of $1^\infty$ and the constants $c,C$ depend only on the degree of the tree.

We will also use \cite[Proposition 3.8]{Brieussel2013} that the size of the inverted orbit is the sum of that of the sections  $|\O(X_n)|=\sum_{x=1}^{d}|\O(X_n|_x)|$. By induction, we deduce that for any level $\ell$
\begin{eqnarray}\label{invosum}
|\O(X_n)|=\sum_{v \in X^\ell}|\O(X_n|_v)|.
\end{eqnarray}

Let us first understand a section $X_n|_x$ of the first level. By \cite[Lemma 4.1]{Brieussel2013}, for each $1 \le x \le d$, the random word on the section at $x$ has the form $X_n|_x=b^x_1a^x_1b^x_2\dots b^x_sa^x_s$ where the factors $a^x_r$ are uniform in $A$ and the factors $b^x_r$ are uniform in $\L H$. These factors are all independent and the number $s$ follows a binomial law $\mathcal B(n,\frac{d-1}{d^2})$. In particular, the $\u_{\Lambda H}\ast \u_A$ random walk induces (slowed down) $\u_{\Lambda H}\ast \u_A$ random walks on the sections. The key point is that the refreshing parameter $\r$ increases under taking sections.

To see this, let us denote an increment of $X_n$ by $s_i=b_ia_i$ with $a_i$ uniform in $A$ and $b_i=\langle \l_i\d_{1^\infty}h_i,a_{i2},\dots ,a_{id_0}\rangle \pi_i$ uniform in $\L H$. Then each factor $b^x_r$ of the section $X_n|_x$ is a product of exactly $k$ factors of the form $\l_i\d_{1^\infty}h_i$ with probability $\left(\frac{1}{d}\right)^k\frac{d-1}{d}$. Similarly each factor $a_r^x$ is a product of exactly $k$ factors $a_{ij}$ with probability $\left(\frac{d-1}{d}\right)^k\frac{1}{d}$. (Note that this is not true for the last factor of $X_n|_x$ because of time truncation.) This is so because the independent factors $a_i$ resample for successive increments which (unique) section will receive the increment in $\Lambda H$ (the one at $1.X_{i-1}^{-1}$), while the $d-1$ other sections will receive an increment in $A$.  

It follows that a factor $b_r^x$ is refreshed with probability
\[
\r_1=\sum_{k=1}^\infty (1-(1-\r)^k)\left(\frac{1}{d}\right)^k\frac{d-1}{d}=1-\frac{1-\r}{1+\frac{\r}{d-1}}>\r,
\]
and a similar formula holds for the refreshing parameter of $a_r^x$. By induction, we deduce that for any initial $\r>0$, there exists a level $\ell$ at which the refreshing parameter is $\r_\ell\ge \frac{1}{2}$. 

We are now ready to consider the refreshed sample $Y_n^\r$. As in the proof of Lemma \ref{lemma:Hinv}, we consider the word $Y_n'^{\rm{word}}$ obtained by resampling the locations of the factors and the $\mu_G=\u_H\ast \u_A$ part of the refreshed samples (but not yet the $\u_\Lambda$ part). The conditioned entropy $\H(Y_n^\r|Y_n'^{\rm{word}})$ is bounded below by the number of sites in $S$ where the lamp will be refreshed. Using the partition~(\ref{invosum}), we first count the number of points of $\O(Y_n'^{\rm{word}}|_v) \subset \partial vT$ where the lamps are refreshed. By the discussion above and the proof of Lemma \ref{lemma:Hinv}, the expected number of refreshed sites in $S \cap \partial vT$ is at least $\r_\ell \E|\O(X_n|_v)|$.

Then:
\[
\H(Y_n^\r|X_n) \ge \H(Y_n^\r|Y_n'^{\rm{word}}) \ge \sum_{v \in X^\ell} \r_\ell \H(\u_\Lambda)\E|\O(X_n|_v)| \ge \frac{1}{2}\H(\u_\l) \E|\O(X_n)| \ge \frac{c}{2}\H(\u_\Lambda)\H(X_n)
\]
using successively the independence of lamp resampling at disjoint sites, (\ref{invosum}) and (\ref{AV4}).
\end{proof}

The choice of measure is heavily used in the proof to get similar random walks at the sections. 

\begin{remark}
It would be interesting to have a similar result for the groups $G(A,H)$ themselves rather than such permutational product extensions, for instance for the first Grigorchuk group. However it is difficult to describe precisely the effect of noise on the $\mu_{HA}$ part. Indeed, not only does this resample the increments in one section but it also shuffles the increments between different sections of a given level. Heuristically, this means that the effect of noise in the sections is even stronger than simply ``increasing the parameter'', as was used in the proof above. We should also mention that the precise asymptotic behaviour of the entropy of (even simple) random walks on the first Grigorchuk group is still unknown.
\end{remark}

\section{Perspectives and questions}\label{sec:qu}

As mentionned in the introduction, our observations on noise sensitivity lead us to believe that the only obstructions to noise sensitivity are virtual homomorphisms onto $\Z$ and non-Liouville property, whence Question  \ref{conj:ell1} and Conjectures \ref{conj:Grig} and \ref{conj:entropy}. We record here some questions and tasks for further study of noise sensitivity of groups.

{\bf 1.} Find more examples of $\ell^1$-noise sensitive groups.
\begin{itemize}
\item[$\bullet$] Clarify which virtually abelian (or more generally nilpotent) groups are $\ell^1$-noise sensitive, and for which measures.
\item[$\bullet$] Is the wreath product of a finite group with the dihedral group $\ell^1$-noise sensitive for some measure? If yes, it would provide an example with exponential growth.
\item[$\bullet$] Study noise sensitivity phenomena in other Liouville groups, such as degree $0$ automata groups \cite{AAMBV}, degree $1$ mother automata groups \cite{AAV} or simple groups \cite{MB,Nekrashevych}.
\item[$\bullet$] Find examples of \emph{strongly} $\ell^1$-noise sensitive groups, i.e. noise sensitive with respect to any (finitely supported generating) probability measure. Possibly, this would be the case for any torsion or simple Liouville group. We expect this property to hold for the first Grigorchuk group.
\end{itemize}

{\bf 2.} Noise sensitivity could also be studied \emph{quantitatively}.
\begin{itemize}
\item[$\bullet$] The choice of  refreshing parameter $\r$ to be constant is arbitrary and it is natural to consider a parameter $\r(n)$ depending on the length $n$. There should be a threshold over which the noised random walk resembles an independent sample. Our definitions of noise sensitivity simply require that this threshold is tending to $0$. 

A lower bound is given by entropy consideration: the entropy of the noise should be no less than that of the independent sample: $\r(n)n\H(\mu)\ge \H(\mu_n)$. However this is not enough in general (e.g. $\Z$ is not $\ell^1$-noise sensitive).

\item[$\bullet$] Does partial $\ell^1$-noise sensitivity ($\exists c >0, \forall \r\in(0,1), \liminf \|\pi_n^\r-\mu_n^2\|_1\ge c$) imply Liouville property?
\end{itemize}

{\bf 3.} About entropy noise sensitivity:
\begin{itemize}
\item[$\bullet$] it is likely that the proof of Theorem \ref{thm:lamplighter} could be improved to show that  when $G$ is entropy noise sensitive, then $G\wr \Z$ and $G \wr \Z^2$ are entropy noise sensitive as well. It would be a consequence of Conjecture \ref{conj:entropy}.
By Proposition~\ref{prop:enstoLiouville}, this is no longer true for wreath products with $\Z^d$ for $d\ge 3$ as they are non-Liouville. 
\item[$\bullet$] Is it true that partial entropy noise sensitivity is equivalent to entropy noise sensitivity?
\end{itemize}

{\bf 4.} About the relationship between notions of noise sensitivity:
\begin{itemize}
\item[$\bullet$] Is it true that $\ell^1$-noise sensitivity implies entropy noise sensitivity? 
\item[$\bullet$] Is it true that large scale noise sensitivity implies  noise sensitivity in average word distance?
\item[$\bullet$] Are there random walks on groups not satisfying the homogeneity assumptions of Definitions \ref{def:ell1toentropy} and \ref{def:ell1todist}?
\end{itemize}

{\bf 5.} About average distance noise sensitivity.
\begin{itemize}
\item[$\bullet$] Abelian groups are not (even partially) noise sensitive in average distance for any measure. Are there other groups with this property ?
\item[$\bullet$] Is it possible that $\liminf\frac{\E d_G(X_n,Y_n^\r)}{\E d_G(X_n,X'_n)} > 1$?
\end{itemize}

{\bf Acknowledgement.} We thank anonymous referees for useful comments. We also thank Gidi Amir and Ryokichi Tanaka for interesting discussions regarding noise sensitivity. We are especially grateful to Gidi Amir and an anonymous referee for pointing out an error in a previous version, and contribution to repair it.

\bibliographystyle{abbrv}
\bibliography{NS}

\end{document}